\documentclass{article}

\usepackage{PRIMEarxiv}

\usepackage[utf8]{inputenc}
\usepackage[T1]{fontenc}
\usepackage[hidelinks]{hyperref}
\usepackage{url}
\usepackage{booktabs}
\usepackage{array}
\usepackage{amsmath}
\usepackage{amssymb}
\usepackage{amsthm}
\usepackage{amsfonts}
\usepackage{nicefrac}
\usepackage{natbib}
\usepackage{microtype}
\usepackage{fancyhdr}
\usepackage{graphicx}
\usepackage{xcolor}
\usepackage{enumitem}
\graphicspath{{figures/}}
\newcommand{\expOneVftErr}{8.0\times 10^{-3}}

\newcommand{\expOneBaseErr}{8.3\times 10^{-3}}
\newcommand{\expOneBaseInRangeErr}{6.3\times 10^{-4}}
\newcommand{\expOneGroupVft}{2.4\times 10^{-15}}
\newcommand{\expOneGroupBase}{6.1\times 10^{-2}}

\newcommand{\expOneTimeBase}{9}
\newcommand{\expTwoLamF}{(-0.03, -0.76)}
\newcommand{\expTwoFpF}{(21.81, -0.12)}
\newcommand{\expTwoErrF}{1.6\times 10^{-1}}
\newcommand{\expTwoVfF}{7.6\times 10^{-2}}

\newcommand{\expTwoLamA}{(0.06, -0.81)}
\newcommand{\expTwoFpA}{(1.24, 0.01)}
\newcommand{\expTwoErrA}{2.9\times 10^{-1}}

\newcommand{\expTwoLossA}{2.4\times 10^{-4}}
\newcommand{\expTwoLamJ}{(0.98, -0.98)}
\newcommand{\expTwoFpJ}{(-2.0\times 10^{-4}, 5.9\times 10^{-5})}
\newcommand{\expTwoErrJ}{5.6\times 10^{-3}}
\newcommand{\expTwoVfJ}{4.0\times 10^{-3}}
\newcommand{\expTwoLossJ}{3.8\times 10^{-5}}

\newcommand{\expThreeErrL}{4.2\times 10^{-2}}
\newcommand{\expThreeErrFixed}{1.8\times 10^{0}}
\newcommand{\expThreeErrBase}{5.3\times 10^{0}}
\newcommand{\expThreeErrBaseShort}{9.3\times 10^{-2}}

\newcommand{\expThreeFreqErr}{1.4\times 10^{-1}}
\newcommand{\expThreeFreqErrLarge}{2.5\times 10^{-2}}

\newcommand{\expThreeBErr}{3.0\times 10^{-2}}

\newcommand{\expFourTrueFloquet}{-1.684}
\newcommand{\expFourTrueRate}{0.750}
\newcommand{\expFourTrueOmegaC}{0.885}
\newcommand{\expFourTrueOmegaO}{0.661}
\newcommand{\expFourTruePeriod}{7.096}

\newcommand{\expFourCErr}{2.8\times 10^{0}}
\newcommand{\expFourQErr}{4.0\times 10^{0}}

\newcommand{\expFourCLoss}{6.1\times 10^{-2}}
\newcommand{\expFourQLoss}{4.2\times 10^{-2}}

\newcommand{\expFourBCInv}{-0.159, 0.318, 0.018, 0.041}

\newcommand{\expFourBTrueInv}{0.750, -1.684, 0.661, 0.885}

\newcommand{\expFiveAlpha}{0.998, 0.499, 0.802, 0.402}
\newcommand{\expFiveAlphaErr}{2.3\times 10^{-3}}
\newcommand{\expFiveErr}{2.8\times 10^{-2}}
\newcommand{\expFiveErrShort}{1.2\times 10^{-2}}
\newcommand{\expFiveFp}{(1.997, 1.997)}
\newcommand{\expFiveNoise}{0.02}
\newcommand{\expSixErrS}{7.0\times 10^{-2}}
\newcommand{\expSixErrR}{3.0\times 10^{2}}
\newcommand{\expSixEnergyS}{5.6\times 10^{-3}}

\newcommand{\expSixSympS}{6.4\times 10^{-13}}

\newcommand{\expSixSeeds}{losses $4.9\times 10^{-5}$, $6.3\times 10^{-5}$; errors $9.0\times 10^{-2}$, $2.0\times 10^{-1}$}
\newcommand{\expSevenEigJ}{-0.56+0.84i, -0.56-0.84i, -1.05}
\newcommand{\expSevenEigErrJ}{9.5\times 10^{-1}}
\newcommand{\expSevenErrJ}{8.8\times 10^{-2}}

\newcommand{\expSevenLossJ}{5.6\times 10^{-3}}
\newcommand{\expSevenEigF}{-0.79+0.43i, -0.79-0.43i, -0.80}
\newcommand{\expSevenEigErrF}{4.8\times 10^{-1}}
\newcommand{\expSevenErrF}{1.2\times 10^{-1}}

\newcommand{\expSevenEigB}{-0.47+1.80i, -0.47-1.80i, -1.02}
\newcommand{\expSevenEigErrB}{2.0\times 10^{-1}}
\newcommand{\expSevenErrB}{2.6\times 10^{-2}}

\newcommand{\expSevenLossUniB}{1.2\times 10^{-3}}

\newcommand{\expSevenJacDevSmallB}{5.2\times 10^{-2}}

\newcommand{\expSevenRelTinyB}{5.5\times 10^{-1}}

\newcommand{\expSevenEigR}{-0.50+1.99i, -0.50-1.99i, -1.00}
\newcommand{\expSevenEigErrR}{8.3\times 10^{-3}}
\newcommand{\expSevenErrR}{3.6\times 10^{-3}}
\newcommand{\expSevenErrMaxR}{1.2\times 10^{-2}}
\newcommand{\expSevenRelLR}{1.8\times 10^{-2}}

\newcommand{\expSevenFpR}{3.5\times 10^{-7}}

\newcommand{\expSevenLossUniR}{1.2\times 10^{-3}}

\newcommand{\expSevenJacDevSmallR}{5.6\times 10^{-3}}
\newcommand{\expSevenJacDevMidR}{6.0\times 10^{-2}}
\newcommand{\expSevenJacDevTinyR}{6.6\times 10^{-4}}
\newcommand{\expSevenRelSmallR}{3.9\times 10^{-3}}
\newcommand{\expSevenRelTinyR}{3.9\times 10^{-3}}
\newcommand{\expSevenRelMidR}{6.6\times 10^{-3}}
\newcommand{\expSevenTimeR}{85}
\newcommand{\expSevenItImR}{4\,360}

\newcommand{\expFourCurPlateauA}{3.1\times 10^{-2}}

\newcommand{\expFourLatLoss}{7.4\times 10^{-1}}

\newcommand{\expFourLatLossLatent}{2.5\times 10^{-2}}

\newcommand{\expFourDLossLagHalf}{4.1\times 10^{-3}}

\newcommand{\expFourDLossLagFour}{1.1\times 10^{0}}
\newcommand{\expFourDLossOneStep}{4.3\times 10^{-4}}
\newcommand{\expFourDErr}{4.2\times 10^{0}}

\newcommand{\expFourDInv}{0.282, -0.832, -0.148, -0.178}

\newcommand{\expFourDRunMeanMin}{4.0\times 10^{-3}}

\newcommand{\expFourDRunMeanRatio}{80}

\newcommand{\expFourEEstInv}{0.727, -1.599, -0.647, -0.887}
\newcommand{\expFourEEstErrMax}{5}
\newcommand{\expFourEEstErrMin}{0.2}

\newcommand{\expFourEFInvErrRate}{13}
\newcommand{\expFourEFInvErrFloquet}{25}

\newcommand{\expFourEFInvErrWcDec}{0.2}

\newcommand{\expFourGErrInBox}{1.7\times 10^{-1}}

\newcommand{\expFourGErrInBoxFinal}{7.3\times 10^{-2}}

\newcommand{\expFourGVfInner}{7.1\times 10^{-2}}
\newcommand{\expFourGVfMedian}{3.6\times 10^{-2}}

\newcommand{\expFourStrataNear}{1.4}
\newcommand{\expFourStrataTrans}{7.0}

\newcommand{\expFourGLossLagHalf}{3.1\times 10^{-4}}
\newcommand{\expFourGLossLagOne}{3.9\times 10^{-4}}
\newcommand{\expFourGLossLagTwo}{6.1\times 10^{-4}}
\newcommand{\expFourGLossLagFour}{2.1\times 10^{-3}}
\newcommand{\expFourGLossOneStep}{2.2\times 10^{-4}}
\newcommand{\expFourGErr}{7.3\times 10^{-1}}

\newcommand{\expFourGVf}{5.8\times 10^{0}}
\newcommand{\expFourGInv}{0.738, -1.400, -0.743, -0.885}

\newcommand{\expFourGInvErrRate}{2}
\newcommand{\expFourGInvErrFloquet}{17}
\newcommand{\expFourGInvErrWo}{12}

\newcommand{\expFourGInvErrWcDec}{0.03}

\newcommand{\expFourGLInv}{1.214, -0.599, -1.339, -0.859}

\newcommand{\expFourGLInvErrRate}{62}

\newcommand{\expFourGLInvErrWo}{102}

\newcommand{\expEightEigP}{-0.45+0.71i, -0.45-0.71i, -0.35, -0.49, -0.62, -1.06}
\newcommand{\expEightEigErrP}{2.0\times 10^{0}}

\newcommand{\expEightErrP}{1.5\times 10^{-1}}

\newcommand{\expEightLossUniP}{2.4\times 10^{-2}}

\newcommand{\expEightJacDevMidP}{1.2\times 10^{0}}

\newcommand{\expEightRelTinyP}{2.6\times 10^{0}}

\newcommand{\expEightTrueEig}{-0.50+2.00i, -0.50-2.00i, -1.00+1.00i, -1.00-1.00i, -0.30, -2.00}
\newcommand{\expEightEigR}{-0.49+1.96i, -0.49-1.96i, -1.00+1.00i, -1.00-1.00i, -0.30, -1.95}
\newcommand{\expEightEigErrR}{4.6\times 10^{-2}}
\newcommand{\expEightEigErrRelR}{2.3}
\newcommand{\expEightErrR}{1.8\times 10^{-2}}
\newcommand{\expEightErrMaxR}{7.1\times 10^{-2}}
\newcommand{\expEightRelLR}{7.9\times 10^{-2}}

\newcommand{\expEightFpR}{4.9\times 10^{-6}}

\newcommand{\expEightLossUniR}{2.0\times 10^{-2}}

\newcommand{\expEightJacDevTinyR}{2.6\times 10^{-3}}
\newcommand{\expEightJacDevMidR}{1.1\times 10^{-1}}

\newcommand{\expEightRelTinyR}{1.9\times 10^{-2}}
\newcommand{\expEightRelMidR}{2.0\times 10^{-2}}

\newcommand{\expEightTimeR}{108}
\newcommand{\expEightParamsR}{56\,634}
\newcommand{\expNineErrS}{5.0\times 10^{-2}}
\newcommand{\expNineErrMaxS}{1.8\times 10^{-1}}
\newcommand{\expNineErrShortS}{2.1\times 10^{-2}}
\newcommand{\expNineEnergyS}{4.6\times 10^{-3}}

\newcommand{\expNineSympS}{3.9\times 10^{-13}}

\newcommand{\expNineVfS}{5.7\times 10^{-3}}
\newcommand{\expNineLossS}{8.5\times 10^{-5}}
\newcommand{\expNineFreqErrAS}{0.2}
\newcommand{\expNineFreqErrBS}{0.3}
\newcommand{\expNineFreqErrMaxS}{3}

\newcommand{\expNineActMedAS}{0.2}
\newcommand{\expNineActMedBS}{0.2}
\newcommand{\expNineFreqMedAS}{0.05}
\newcommand{\expNineFreqMedBS}{0.08}
\newcommand{\expNineTimeS}{56}
\newcommand{\expNineParamsS}{35\,376}

\newcommand{\expNineErrR}{9.7\times 10^{-2}}

\newcommand{\expNineEnergyR}{4.7\times 10^{-2}}

\newcommand{\expNineSympR}{4.4\times 10^{1}}

\newcommand{\expNineLossR}{5.5\times 10^{-4}}

\newcommand{\expNineActMedAR}{10.2}
\newcommand{\expNineActMedBR}{26.8}
\newcommand{\expNineFreqMedAR}{0.27}
\newcommand{\expNineFreqMedBR}{0.28}
\newcommand{\expNineTimeR}{34}
\newcommand{\expNineParamsR}{36\,548}

\newcommand{\expNineErrW}{3.5\times 10^{0}}

\newcommand{\expNineEnergyW}{1.3\times 10^{10}}

\newcommand{\expNineLossW}{1.2\times 10^{-1}}

\newcommand{\expOneRelErr}{1.1\times 10^{-2}}

\newcommand{\expOneRelParams}{77\,638}
\newcommand{\expOneDeepEqLoss}{7.6\times 10^{-5}}

\newcommand{\expOneDeepErr}{5.1\times 10^{-3}}
\newcommand{\expOneDeepMax}{1.9\times 10^{-2}}
\newcommand{\expOneDeepTime}{103}
\newcommand{\expOneDeepParams}{53\,094}

\newcommand{\expSevenSeedsEigErr}{8.3\times 10^{-3} / 8.0\times 10^{-3} / 8.0\times 10^{-3}}
\newcommand{\expSevenSeedsEigErrMax}{8.3\times 10^{-3}}
\newcommand{\expSevenSeedsErr}{3.6\times 10^{-3} / 3.6\times 10^{-3} / 4.1\times 10^{-3}}

\newcommand{\expSevenSeedsFp}{3.5\times 10^{-7} / 1.1\times 10^{-6} / 8.2\times 10^{-7}}
\newcommand{\expSevenSeedsSpectra}{$\{-0.495-1.993i, -0.495+1.993i, -1.000\}$; $\{-0.493-1.996i, -0.493+1.996i, -1.003\}$; $\{-0.493-1.997i, -0.493+1.997i, -1.002\}$}
\newcommand{\expSevenSeedsTime}{85 / 59 / 30}

\newcommand{\expEightSeedsEigErr}{4.6\times 10^{-2} / 4.5\times 10^{-2} / 2.7\times 10^{-2}}

\newcommand{\expEightSeedsEigRel}{2.3 / 2.2 / 1.3}
\newcommand{\expEightSeedsEigRelMax}{2.3}
\newcommand{\expEightSeedsErr}{1.8\times 10^{-2} / 1.5\times 10^{-2} / 1.6\times 10^{-2}}

\newcommand{\expEightSeedsFp}{4.9\times 10^{-6} / 1.5\times 10^{-6} / 3.9\times 10^{-6}}
\newcommand{\expEightSeedsTime}{108 / 81 / 81}

\newcommand{\expNineSeedsLoss}{8.5\times 10^{-5} / 1.2\times 10^{-4} / 9.1\times 10^{-5}}
\newcommand{\expNineSeedsErr}{5.0\times 10^{-2} / 5.8\times 10^{-2} / 4.2\times 10^{-2}}

\newcommand{\expNineSeedsEnergy}{4.6\times 10^{-3} / 4.9\times 10^{-3} / 5.6\times 10^{-3}}

\newcommand{\expNineSeedsSymp}{3.9\times 10^{-13} / 5.1\times 10^{-13} / 3.9\times 10^{-13}}

\newcommand{\expNineSeedsFreq}{0.05, 0.08 / 0.09, 0.09 / 0.07, 0.09}
\newcommand{\expNineSeedsFreqMax}{0.09}
\newcommand{\expNineSeedsAct}{0.15, 0.21 / 0.22, 0.16 / 0.17, 0.18}
\newcommand{\expNineSeedsActMax}{0.22}
\newcommand{\expNineSeedsTime}{56 / 45 / 44}

\newcommand{\expFourCmpOStep}{1.9\times 10^{-3}}

\newcommand{\expFourCmpOLagHalf}{9.9\times 10^{-4}}

\newcommand{\expFourCmpOLagFour}{3.9\times 10^{-3}}
\newcommand{\expFourCmpOInBox}{2.8\times 10^{-1}}

\newcommand{\expFourCmpEStep}{3.9\times 10^{-4}}

\newcommand{\expFourCmpELagHalf}{5.3\times 10^{-4}}

\newcommand{\expFourCmpELagFour}{4.9\times 10^{-3}}
\newcommand{\expFourCmpEInBox}{2.9\times 10^{-1}}

\newcommand{\expFourCmpGStep}{2.2\times 10^{-4}}

\newcommand{\expFourCmpGLagHalf}{3.0\times 10^{-4}}

\newcommand{\expFourCmpGLagFour}{2.2\times 10^{-3}}
\newcommand{\expFourCmpGInBox}{1.7\times 10^{-1}}

\newcommand{\R}{{\mathbb{R}}}
\newcommand{\dd}{\mathrm{d}}
\DeclareMathOperator{\id}{id}
\DeclareMathOperator{\Lip}{Lip}

\newtheorem{theorem}{Theorem}
\newtheorem{proposition}[theorem]{Proposition}
\newtheorem{lemma}[theorem]{Lemma}
\theoremstyle{definition}
\newtheorem{definition}[theorem]{Definition}
\theoremstyle{remark}
\newtheorem{remark}[theorem]{Remark}

\title{Neural Vector Field Transformations: Learning the Flow of Autonomous Systems by Conjugation to Explicitly Integrable Reference Flows}

\author{
  Andrew Qing He\\
  Department of Mathematics \\
  Duke University \\
  Durham, North Carolina, US\\
  \texttt{andrew.hq@duke.edu}
}

\begin{document}
\maketitle

\begin{abstract}
We represent the flow map of an autonomous ODE $\dot x=f(x)$ as a conjugate $\Phi_t=F_\theta^{-1}\circ\Psi_t\circ F_\theta$ of an explicitly integrable \emph{reference flow} $\Psi_t$ by an invertible neural network $F_\theta$, the \emph{neural vector field transformation}. The representation is an exact flow for every parameter value and inherits the phase portrait of the reference. A $C^1$ conjugacy preserves the eigenvalues of equilibria, the periods and Floquet multipliers of periodic orbits and the frequency--amplitude relation of centers, so these invariants must be learnable parameters of $\Psi$. We give parametrized reference flows with explicit integral curves, show that an action--angle reference is Hamiltonian only if its learnable frequency map is a gradient, and show that an orientation-preserving $F_\theta$ cannot reverse the rotation sense of the reference, so it must be fixed before training. The main contribution is a training recipe: a push-forward equation residual made relative to $|f|$ and sampled log-uniformly in $|x|$ around an equilibrium, with the anchors $F_\theta(0)=0$, $DF_\theta(0)=I$; a symplectic transformation with the gradient-form frequency map; and invariants initialized from the field or data. With it the full spectrum of a three- and a six-dimensional nonlinear sink is recovered (eigenvalue deviation $\expSevenSeedsEigErrMax$ and relative error $\expEightSeedsEigRelMax\%$, three seeds each), and the action--angle variables of a four-dimensional integrable Hamiltonian system with frequency maps and actions correct to a median of $\expNineSeedsFreqMax\%$ and $\expNineSeedsActMax\%$, and an exactly symplectic flow. On a dispersive system a time-input network is eight times more accurate inside its horizon, and the data-only identification of the Van der Pol oscillator yields a conjugacy on the data region with the Floquet exponent $\expFourGInvErrFloquet\%$ off; an oracle run with the true invariants frozen attributes this to the expressivity of $F_\theta$.
\end{abstract}


\section{Introduction}
Autonomous ordinary differential equations (ODEs) $\dot x=f(x)$ on an open set $U\subseteq\R^n$ are increasingly treated with learning methods, either to solve a known system without time stepping \citep{lagaris1998artificial,raissi2019pinn}, to identify an unknown system from time series \citep{qin2019data,chen2018neuralode}, or to combine partial physical knowledge with data \citep{wang2021physics}. An autonomous system carries two structures that a general non-autonomous system lacks: a fixed vector field $f$ on the phase space, and a one-parameter group of diffeomorphisms, the \emph{flow} $\{\Phi_t\}_{t\in\R}$ generated by $f$, which satisfies
\begin{equation}\label{eq:group}
\Phi_0=\id,\qquad \Phi_s\circ\Phi_t=\Phi_{s+t}\quad\text{for all } s,t\in\R .
\end{equation}
Neural representations of the flow map usually sacrifice one of two things. A network $\Phi_\theta(t,x)$ with time as an input \citep{bilos2021neuralflows,raissi2019pinn} can be queried at any $t$ but does not satisfy \eqref{eq:group}; a network $\Phi^{(h)}_\theta(x)$ approximating the time-$h$ map \citep{qin2019data,jin2020sympnets} satisfies \eqref{eq:group} on the grid $h\mathbb Z$ only and must be iterated.

This paper studies a representation that keeps both. We write the unknown flow as a conjugate of a simple flow,
\begin{equation}\label{eq:ansatz}
\Phi_t(x)=F^{-1}\big(\Psi_t(F(x))\big),
\end{equation}
where $F$ is a diffeomorphism represented by an invertible neural network $F_\theta$ and $\Psi_t$ is a \emph{reference flow} whose integral curves are available in closed form or by a scalar quadrature, so that $\Psi_t(y)$ costs one function evaluation for any real $t$. Equation \eqref{eq:ansatz} is a flow for every parameter value, so \eqref{eq:group} holds exactly and $t$ is a continuous argument. The motivation is the rectification theorem \citep{arnold1992ode,lee2012smooth,palis1982geometric}: near every regular point a vector field is conjugate to the constant field $e_1$, whose flow is a translation. We call $F$ a \emph{vector field transformation}, the model a \emph{neural vector field transformation} (neural VFT), and $\Psi$ the \emph{reference flow}. The same ansatz with an affine reference was proposed independently as neural conjugate flows by \citet{bizzi2025ncf}; Section~\ref{sec:related} positions the present work relative to that and to the Koopman and conjugate-mapping literature.

Globally, the translation flow has no equilibria, periodic orbits or recurrence, and most systems of interest have all three, so the reference must share the qualitative behavior of the target. The first message of this paper is that \emph{qualitative} is not enough: a $C^1$ conjugacy preserves the eigenvalues of every equilibrium, the period and Floquet multipliers of every periodic orbit and the frequency--amplitude relation around a center. A fixed reference fixes these numbers and no $F$ can then make \eqref{eq:ansatz} hold. The reference must carry learnable parameters that expose all of its conjugacy invariants, and a large family of such reference flows still has explicit integral curves. The second message, which occupies most of the paper, is that whether the method works is decided by whether the transformation can actually be fitted, and that this depends on details of the loss, the sampling and the initialization that the theory does not dictate. We therefore state our contributions at the level the evidence supports.

\paragraph{Contributions.}
\begin{enumerate}[nosep,leftmargin=*]
\item \emph{Theory} (Section~\ref{sec:ansatz}). The neural VFT is a flow for every parameter value and inherits the phase portrait of the reference (Proposition~\ref{prop:flow}); a global trivialization to a translation exists exactly for dispersive flows (Proposition~\ref{prop:global}); a reference must match the $C^1$-conjugacy invariants of the target (Proposition~\ref{prop:invariants}) and, since the transformations we train preserve orientation, its sense of rotation (Remark~\ref{rem:orientation}); parametrized reference flows with explicit integral curves expose these invariants, and the action--angle reference with $d\geq2$ degrees of freedom is Hamiltonian if and only if its frequency map is a gradient (Lemma~\ref{lem:gradient}); the non-uniqueness of $F_\theta$ is the centralizer of $\Psi$, and anchors remove it (Proposition~\ref{prop:centralizer}).
\item \emph{Training methodology} (Section~\ref{sec:training}). Three components were confirmed experimentally across three seeds each: the push-forward equation residual, made \emph{relative} to $|f|$ and evaluated on collocation points whose radius around the equilibrium is log-uniformly distributed, together with the point anchors $F_\theta(0)=0$ and $DF_\theta(0)=I$; a symplectic transformation with the gradient-form frequency map; and initialization of the invariants of the reference, in particular the sense of rotation, from the field or from the data. We also give the diagnostics that distinguish a conjugacy from an interpolant: the relative residual on spheres around the equilibrium should be flat across decades of $|x|$, $\|DF_\theta(x)-I\|$ should grow linearly in $|x|$, and the true trajectories mapped by $F_\theta$ should be orbits of the reference.
\item \emph{Experiments} (Section~\ref{sec:experiments}). The full spectrum of a three-dimensional and of a six-dimensional nonlinear sink is recovered, and the action--angle variables of a four-dimensional integrable Hamiltonian system are recovered with an exactly symplectic learned flow; these three results were repeated with three seeds (Table~\ref{tab:seeds}). Single-seed results include the frequency--amplitude relation of the pendulum, the identification of all four Lotka--Volterra parameters from noisy data to $\expFiveAlphaErr$ relative error, an exactly symplectic Duffing flow, and the spectrum of a planar saddle whose stability type had to be discovered during training. Two negative results are stated with the same prominence: on a dispersive planar system a time-input network is more accurate than the transformation inside its horizon, and the data-only Van der Pol identification remains open.
\item \emph{Failure modes} (Section~\ref{sec:failures}). We document four mechanisms by which training fails and what removed each: the bifurcation barrier when the stability type of the reference must change; degenerate-Jacobian solutions that hide an invariant mismatch in a neighborhood of the equilibrium too small to be sampled, which circumvent a pointwise Jacobian anchor; the orientation barrier; and the data-only limit cycle, where an oracle run locates the obstacle in the transformation.
\end{enumerate}

\section{The conjugation ansatz}\label{sec:ansatz}

\subsection{Flows, conjugacy and the rectification theorem}
Throughout, $U\subseteq\R^n$ is open and $f\colon U\to\R^n$ is $C^1$. We write $\Phi_t(x)$ for the solution of $\dot x=f(x)$, $x(0)=x$, at time $t$. Two flows $\Phi$ on $U$ and $\Psi$ on $V$ are \emph{$C^k$-conjugate} if there is a $C^k$ diffeomorphism $F\colon U\to V$ with $F\circ\Phi_t=\Psi_t\circ F$ for all $t$; in terms of the generators $f$ and $g$ this is $F_*f=g$, i.e.
\begin{equation}\label{eq:pushforward}
DF(x)\,f(x)=g(F(x))\qquad\text{for all }x\in U .
\end{equation}
\begin{theorem}[Rectification theorem; {\citealp[\S 7]{arnold1992ode}}, {\citealp[Thm.\ 9.22]{lee2012smooth}}]\label{thm:rect}
If $f(p)\neq0$ there is a neighborhood $U_p$ of $p$ and a $C^1$ diffeomorphism $F\colon U_p\to F(U_p)\subseteq\R^n$ with $F_*f=e_1$ on $U_p$, i.e.\ $\Phi_t(x)=F^{-1}(F(x)+te_1)$ whenever both sides are defined.
\end{theorem}
In coordinates, \eqref{eq:pushforward} with $g=e_1$ reads $\nabla F_1\cdot f=1$ and $\nabla F_i\cdot f=0$ for $i\geq2$: the last $n-1$ components of such a transformation are independent first integrals and the first component is a \emph{time function} increasing at unit rate along orbits. A region containing an equilibrium or a periodic orbit admits no continuous time function, which is the obstruction to a global translation reference.

\subsection{The neural VFT}
\begin{definition}[Reference flow]
A \emph{reference flow} on $V\subseteq\R^n$ is a $C^1$ flow $\{\Psi_t\}$ with generator $g$ for which $\Psi_t(y)$ can be evaluated for arbitrary $t\in\R$ and $y\in V$ by an explicit formula, possibly involving the inversion of a scalar monotone function. It may depend on a finite-dimensional parameter $\beta$.
\end{definition}
\begin{definition}[Neural VFT]
Given a reference flow $\Psi$ on $\R^n$ and an invertible neural network $F_\theta\colon\R^n\to\R^n$ with an exactly computable inverse, the \emph{neural vector field transformation} is the family
\begin{equation}\label{eq:vft}
\Phi^\theta_t:=F_\theta^{-1}\circ\Psi_t\circ F_\theta,\qquad t\in\R .
\end{equation}
\end{definition}
\begin{proposition}[Structure of the learned flow]\label{prop:flow}
For every $\theta$ and $\beta$, $\{\Phi^\theta_t\}$ is a $C^1$ flow on $\R^n$ with generator $f_\theta(x)=DF_\theta(x)^{-1}g(F_\theta(x))$. Moreover:
\begin{enumerate}[nosep,leftmargin=*]
\item[(i)] $S$ is invariant under $\Psi$ if and only if $F_\theta^{-1}(S)$ is invariant under $\Phi^\theta$; $x^*$ is an equilibrium of $\Phi^\theta$ iff $F_\theta(x^*)$ is one of $\Psi$, and $\gamma$ is a periodic orbit of $\Phi^\theta$ of period $T$ iff $F_\theta(\gamma)$ is one of $\Psi$ with the same period;
\item[(ii)] $Df_\theta(x^*)=DF_\theta(x^*)^{-1}Dg(F_\theta(x^*))DF_\theta(x^*)$ at an equilibrium, and the monodromy matrices of corresponding periodic orbits are similar;
\item[(iii)] if $H$ is a first integral of $\Psi$ then $H\circ F_\theta$ is a first integral of $\Phi^\theta$;
\item[(iv)] if $n=2d$, $F_\theta$ is symplectic \emph{and $\Psi$ is the Hamiltonian flow of some $H_0$}, then $\Phi^\theta$ is the Hamiltonian flow of $H_0\circ F_\theta$.
\end{enumerate}
\end{proposition}
\begin{proof}
$\Phi^\theta_0=\id$ and $\Phi^\theta_s\circ\Phi^\theta_t=F_\theta^{-1}\Psi_sF_\theta F_\theta^{-1}\Psi_tF_\theta=F_\theta^{-1}\Psi_{s+t}F_\theta$. Differentiating $F_\theta\circ\Phi^\theta_t=\Psi_t\circ F_\theta$ at $t=0$ gives $DF_\theta\,f_\theta=g\circ F_\theta$. (i) is immediate from \eqref{eq:vft}; (ii) follows by differentiating $f_\theta$ at an equilibrium and $\Phi^\theta_T$ at a point of $\gamma$; (iii) $H(F_\theta(\Phi^\theta_tx))=H(\Psi_t(F_\theta x))=H(F_\theta x)$; (iv) the pull-back of a Hamiltonian vector field by a symplectic map is the Hamiltonian vector field of the pulled-back Hamiltonian, $F_\theta^*X_{H_0}=X_{H_0\circ F_\theta}$ \citep[\S 40]{arnold1989mathematical}.
\end{proof}
Proposition~\ref{prop:flow} is the inductive bias of the method: whatever is trained, the learned model is a flow whose phase portrait is a diffeomorphic copy of that of $\Psi$. It can never acquire spurious equilibria or lose a limit cycle, and it conserves $H\circ F_\theta$ exactly for every conserved quantity $H$ of the reference. The hypothesis on $\Psi$ in (iv) is not a formality: Lemma~\ref{lem:gradient} below gives a natural reference whose flow is explicit but not Hamiltonian, and a symplectic $F_\theta$ does not repair that. The flip side of the proposition is that the model can only represent systems whose phase portrait \emph{is} a diffeomorphic copy of that of $\Psi$, which Section~\ref{sec:invariants} makes quantitative.

\subsection{When does a global trivialization exist?}
A flow on $U$ is \emph{parallelizable} if there is a closed set $\Sigma\subset U$ (a global cross-section) such that $(\sigma,t)\mapsto\Phi_t(\sigma)$ is a homeomorphism $\Sigma\times\R\to U$; it is \emph{dispersive} if any two points have neighborhoods $N_x,N_y$ and a $T>0$ with $\Phi_t(N_x)\cap N_y=\emptyset$ for $|t|>T$.
\begin{proposition}[{\citealp[Thm.\ IV.1.8, IV.2.6]{bhatia1970stability}}]\label{prop:global}
Let $U$ be invariant and the flow complete on $U$. The flow is parallelizable on $U$ if and only if it is dispersive. If it is parallelizable with a $C^1$ cross-section $\Sigma$ transversal to $f$, the map $F(x)=(\tau(x),\xi(x))$, where $\Phi_{-\tau(x)}(x)\in\Sigma$ and $\xi$ is a chart of $\Sigma$, is a $C^1$ diffeomorphism $U\to\R\times\xi(\Sigma)$ with $F_*f=e_1$.
\end{proposition}
The translation reference is therefore admissible exactly for dispersive dynamics; any recurrence requires a different reference. Since $F_\theta\colon\R^n\to\R^n$ is a diffeomorphism of the whole space, $\Phi^\theta$ is complete; a system with finite-time blow-up cannot be represented on all of $\R^n$ but can be on any bounded invariant sub-domain, which is all that training on a box asks for.

\subsection{Conjugacy invariants and orientation}\label{sec:invariants}
\begin{proposition}[Invariants of $C^1$ conjugacy]\label{prop:invariants}
Suppose $F\circ\Phi_t=\Psi_t\circ F$ on $U$ with $F$ a $C^1$ diffeomorphism. Then
\begin{enumerate}[nosep,leftmargin=*]
\item[(a)] $F$ maps equilibria to equilibria and the Jacobians $Df(x^*)$ and $Dg(F(x^*))$ are similar; in particular their spectra coincide;
\item[(b)] $F$ maps periodic orbits to periodic orbits of the same period with similar monodromy matrices; in particular the Floquet multipliers coincide;
\item[(c)] $F$ maps every invariant set, first integral and $\omega$-limit set to its counterpart; the dependence of the period on the orbit (e.g.\ on the energy around a center) is preserved;
\item[(d)] if $F$ is only a homeomorphism, (a) and (b) weaken to: equilibria, their stability type, periodic orbits and their periods are preserved, but spectra and multipliers need not be.
\end{enumerate}
\end{proposition}
\begin{proof}
(a)--(c) are Proposition~\ref{prop:flow}(i)--(iii) read in the opposite direction. (d): periods are preserved because $F\circ\Phi_T=\Psi_T\circ F$; the Hartman--Grobman theorem \citep{hartman1960lemma,palis1982geometric} shows that a $C^0$ conjugacy does not see the spectrum beyond the stable and unstable dimensions.
\end{proof}
\begin{remark}[Consequences for the choice of $\Psi$]\label{rem:consequences}
A fixed reference $e^{At}$ is admissible for a target with a hyperbolic equilibrium of spectrum $\{\lambda_i\}$ only if $\sigma(A)=\{\lambda_i\}$. A target with an attracting limit cycle of period $T$ and Floquet exponent $\mu_F<0$ around an unstable focus with eigenvalues $\lambda_0\pm i\omega_0$ fixes four independent real numbers that the reference must reproduce. The Hopf normal form $\dot r=r(\lambda_0-r^2)$, $\dot\vartheta=\omega_0$ has $\mu_F=-2\lambda_0$ and $T=2\pi/\omega_0$, so it is admissible only on a codimension-two set; the Van der Pol oscillator with $\mu=1.5$ has $\lambda_0=\expFourTrueRate$ and $\mu_F=\expFourTrueFloquet\neq-2\lambda_0$ and is not in that set. Resonances add a subtler restriction: by Poincar\'e--Dulac and Sternberg \citep{sternberg1957local,arnold1988geometrical} a smooth linearization near an equilibrium may fail when $\lambda_j=\sum_im_i\lambda_i$ with $|m|\geq2$, in which case the reference must contain the resonant monomials; a $C^1$ linearization exists under weaker non-resonance conditions, which is the regularity we need.
\end{remark}
\begin{remark}[Orientation]\label{rem:orientation}
Let $n=2$ and let $x^*$ be a focus, i.e.\ $Df(x^*)$ has non-real eigenvalues. The quadratic form $x\mapsto x\wedge Df(x^*)x$ is then definite, and its sign is the sense of rotation of the linearized flow. If $F$ conjugates $\Phi$ to $\Psi$ then $Dg(F(x^*))=P\,Df(x^*)P^{-1}$ with $P=DF(x^*)$, and $(Py)\wedge(PAy)=\det(P)\,y\wedge Ay$, so the sense of rotation is preserved when $\det DF(x^*)>0$ and reversed when $\det DF(x^*)<0$. Likewise an orientation-preserving homeomorphism of the plane preserves the sense in which a periodic orbit is traversed. In $2d$ dimensions with a symplectic $F$ the same holds for each elliptic frequency separately, since the sign of $\omega_i$ relative to the symplectic form (the Krein signature) is a symplectic invariant \citep{arnold1989mathematical}. Affine coupling networks have $\det DF_\theta=\exp(\sum_k s_k)>0$ and symplectic networks have $\det DF_\theta=1$, so every transformation we train is orientation-preserving for every $\theta$. Consequently the sense of rotation of the reference, i.e.\ the sign of its frequency parameters, is not learnable by $F_\theta$: it must be chosen to match the target, and it can be read off $Df(x^*)$ when $f$ is known or off the mean signed increment of the polar angle along the data when it is not. If it is initialized wrongly, the optimizer can only correct it by driving the frequency through zero, where the reference momentarily does not rotate at all; Section~\ref{sec:failures} shows that this \emph{orientation barrier} is as real in practice as the stability barrier of Remark~\ref{rem:consequences}.
\end{remark}

\subsection{Reference flows with learnable invariants}\label{sec:reference}
The requirement is that $\Psi_t(y)$ be explicit for every $t$ while all conjugacy invariants are free parameters. Two constructions give large families of such flows.

\paragraph{Skew products over first integrals.}
Suppose the phase space splits as $y=(\vartheta,I)$ and the reference field is triangular, $\dot I=0$, $\dot\vartheta=h(\vartheta,I)$, with the $\vartheta$-dynamics explicitly integrable for each fixed $I$. Since $I$ is constant along orbits, $\Psi_t(\vartheta,I)=(\chi^I_t(\vartheta),I)$ with $\chi^I$ the explicit flow at frozen $I$, and the dependence on $I$ is arbitrary because it is never integrated. This produces the \textbf{translation} $\Psi_t(y)=y+te_1$ ($\vartheta=y_1$, $I=y_{2:n}$), \textbf{shear flows} $\dot y_1=\lambda(y_{2:n})y_1$, and the \textbf{action--angle flow} on $\R^{2d}$: with $y=(y_{2i-1},y_{2i})_{i=1}^d$ and $I_i=\tfrac12(y_{2i-1}^2+y_{2i}^2)$, each pair is rotated by the angle $\omega_i(I)\,t$ (counter-clockwise for $\omega_i>0$), where $\omega\colon\R^d\to\R^d$ is a learnable map. For $d=1$ this is a nonlinear center with an arbitrary frequency--amplitude relation, and it is the Hamiltonian flow of $-\int\omega(I)\,\dd I$ \citep[\S 50]{arnold1989mathematical}. For $d\geq2$ the flow is closed-form for any $\omega$, but:
\begin{lemma}[Hamiltonian action--angle references]\label{lem:gradient}
The action--angle flow with frequency map $\omega\colon[0,\infty)^d\to\R^d$ is the Hamiltonian flow of some $H_0$ with respect to the standard symplectic form $\sum_i\dd y_{2i-1}\wedge\dd y_{2i}$ if and only if $\omega$ is a gradient, $\omega=\nabla_I\mathcal K$; then $H_0=-\mathcal K(I)$.
\end{lemma}
\begin{proof}
The generator is $X=\sum_i\omega_i(I)\,(-y_{2i}\,\partial_{y_{2i-1}}+y_{2i-1}\,\partial_{y_{2i}})$. With $q_i=y_{2i-1}$, $p_i=y_{2i}$, Hamilton's equations $\dot q_i=\partial H_0/\partial p_i$, $\dot p_i=-\partial H_0/\partial q_i$ hold iff $\dd H_0=-\sum_i\omega_i(I)\,(q_i\dd q_i+p_i\dd p_i)=-\sum_i\omega_i(I)\,\dd I_i$. This one-form is exact iff it is closed, i.e.\ iff $\partial\omega_i/\partial I_j=\partial\omega_j/\partial I_i$ on the convex set $[0,\infty)^d$, i.e.\ iff $\omega$ is a gradient; then $H_0=-\mathcal K$.
\end{proof}
We therefore parametrize $\omega(I)=\omega_0+\nabla_I K(I)$ with a constant vector $\omega_0$ and a scalar network $K$, i.e.\ $\mathcal K=\omega_0\cdot I+K(I)$. Before this was done, a vector-valued network $\omega$ gave a reference that was explicit but not Hamiltonian, and the learned flow failed the symplecticity check even with a symplectic $F_\theta$, exactly as Proposition~\ref{prop:flow}(iv) predicts. The sign of $\omega_0$ is the orientation choice of Remark~\ref{rem:orientation}.

\paragraph{Linear flows.}
$\Psi_t(y)=e^{At}y$ with a full learnable $A\in\R^{n\times n}$, evaluated by a differentiable matrix exponential. The invariants are the spectrum of $A$; the remaining $n^2-n$ parameters are centralizer freedom (Section~\ref{sec:anchor}). The affine flows of \citet{bizzi2025ncf} belong to this class.

\paragraph{Radial quadrature: planar flows with an equilibrium and a limit cycle.}
Let $s=|y|^2$ and prescribe $\dot s=h(s)$, $\dot\vartheta=w(s)$ with $h(0)=h(s^*)=0$ and $h\neq0$ elsewhere on $[0,\infty)$. The radial equation is one-dimensional, hence integrable by quadrature: with $\mathcal T(s)=\int\dd s/h(s)$, $s(t)=\mathcal T^{-1}(\mathcal T(s_0)+t)$ and $\vartheta(t)=\vartheta_0+\int_{s_0}^{s(t)}w/h\,\dd s$. The equilibrium has eigenvalues $\tfrac12h'(0)\pm i\,w(0)$, the cycle has period $2\pi/w(s^*)$ and Floquet exponent $h'(s^*)$. We use the \emph{decoupled} reference $h(s)=s(s^*-s)R(s)$ with $R(s)=R_0+(R_1-R_0)\,4u/(1+u)^2$, $u=s/s^*$, so that $R$ stays between $R_0$ and $R_1$ and $h$ has no further zeros; the origin rate $\lambda_0=s^*R_0/2$, the Floquet exponent $\mu_F=-s^*R_1$, the radius and the frequencies $w(0)$, $w(s^*)$ are independent, as Remark~\ref{rem:consequences} requires. The two logarithmic singularities of $1/h$ are subtracted analytically, the smooth remainder is integrated by a fixed Gauss--Legendre rule, and $\mathcal T$ is inverted on each invariant side of the cycle by bisection followed by Newton steps through which gradients are propagated; the resulting $\Psi_t$ agrees with a high-order ODE solve to $10^{-12}$ and satisfies the group law to rounding. The \emph{coupled} logistic reference $h(s)=s(a-bs)$, $w=w_0+w_1s$, whose Floquet exponent $-a$ is locked to the origin rate $a/2$, is the one Remark~\ref{rem:consequences} rules out for Van der Pol; Section~\ref{sec:failures} records that it indeed gives up its cycle under training.

\subsection{Non-uniqueness and anchors}\label{sec:anchor}
\begin{proposition}[Centralizer]\label{prop:centralizer}
Let $F$ conjugate $\Phi$ to $\Psi$. A diffeomorphism $\tilde F$ also conjugates $\Phi$ to $\Psi$ if and only if $S=\tilde F\circ F^{-1}$ belongs to the centralizer $\mathcal C(\Psi)=\{S: S\circ\Psi_t=\Psi_t\circ S\ \forall t\}$. For the translation, $\mathcal C$ consists of the maps $S(y_1,y')=(y_1+\alpha(y'),\beta(y'))$ with $\beta$ a diffeomorphism of $\R^{n-1}$; for the action--angle flow with $d=1$ and non-constant $\omega$, of the maps $(\vartheta,I)\mapsto(\vartheta+\alpha(I),\beta(I))$; for a linear flow with simple non-resonant spectrum, of the linear maps commuting with $A$.
\end{proposition}
\begin{proof}
$\tilde F\Phi_t\tilde F^{-1}=\Psi_t$ iff $S\Psi_tS^{-1}=\Psi_t$. The descriptions follow by writing $S\circ\Psi_t=\Psi_t\circ S$ in the respective coordinates; in the linear case a $C^1$ map commuting with a hyperbolic non-resonant linear flow is linear \citep{sternberg1957local}.
\end{proof}
The learned flow is the same for every element of the coset, so non-uniqueness is harmless for prediction but harmful for optimization: the centralizer contains unbounded directions (the scalings $\beta$) along which $F_\theta$ can drift or collapse, and $F_\theta$ is not identifiable when it is itself of interest. An \emph{anchor} fixes a representative. For the translation we require $F$ to be the identity on a transversal hyperplane, $F(c,\xi)=(0,\xi)$. For a reference with an equilibrium at $0$ we require $F_\theta(x^*)=0$ and, decisively, $DF_\theta(x^*)=I$ when $x^*$ is known: by Proposition~\ref{prop:flow}(ii) the Jacobian anchor forces the spectrum of the reference to equal that of $Df(x^*)$, which turns an invariant into a locally observable quantity (Section~\ref{sec:training}). For the full linear reference it also removes the $n^2-n$ centralizer parameters of $A$. A symplectic $F_\theta$ cannot scale, so for the action--angle reference the architecture itself removes the scaling direction of the centralizer.

\section{Training methodology}\label{sec:training}
The recipe that the experiments of Section~\ref{sec:experiments} support consists of three components, each of which was decisive in at least one experiment and was confirmed with three seeds:
\begin{enumerate}[nosep,leftmargin=*]
\item[(R1)] the push-forward equation residual, divided by $|f(x)|^2+\varepsilon$, evaluated on collocation points half of which have a log-uniformly distributed radius around the equilibrium, together with the point anchors $F_\theta(0)=0$ and $DF_\theta(0)=I$;
\item[(R2)] for Hamiltonian targets, a symplectic transformation with the action--angle reference in gradient form, $\omega=\omega_0+\nabla K$;
\item[(R3)] initialization of the invariants of the reference, and in particular of the sense of rotation, from the vector field or from the data, never at a default value.
\end{enumerate}
The rest of this section defines the ingredients, explains why each component matters, and gives the diagnostics.

\subsection{Invertible architectures}
Any network with an exactly computable inverse can serve as $F_\theta$: coupling flows \citep{dinh2015nice,dinh2017realnvp}, i-RevNets \citep{jacobsen2018irevnet}, autoregressive flows \citep{kingma2016improving} and, for $n=2d$ and Hamiltonian targets, symplectic networks \citep{jin2020sympnets}. We use affine coupling blocks $y_B=x_B\odot e^{s(x_A)}+t(x_A)$, $y_A=x_A$, with alternating partitions and invertible linear layers initialized at the identity, the scale $s$ bounded by a $\tanh$ of amplitude $s_{\max}$ (the \emph{scale bound}); coupling flows are universal approximators of diffeomorphisms \citep{teshima2020coupling}. In the symplectic case we use the gradient modules $(q,p)\mapsto(q,p+\nabla V(q))$ and $(q,p)\mapsto(q+\nabla T(p),p)$ of SympNets. The scale bound matters: with $s_{\max}=3$ a single block may stretch by $e^3$ and the equation residual, which multiplies the Jacobians of all blocks, is correspondingly stiff; lowering $s_{\max}$ to $1$ and adding blocks broke an optimization plateau in Experiment~1 with fewer parameters.

\subsection{Losses}\label{sec:losses}
Three information regimes occur: (1) $f$ is known; (2) only sequential observations $(t^{(k)},x^{(k)})$ of trajectories are available; (3) $f(\cdot;\alpha)$ is known up to parameters $\alpha$ and observations are available.

\paragraph{Equation residual and its push-forward form.}
Equation \eqref{eq:pushforward} can be imposed as a pull-back residual $|DF_\theta(x)f(x)-g(F_\theta(x))|^2$ or as the push-forward residual
\begin{equation}\label{eq:lpush}
\ell^{\mathrm{eq}}(x)=\big|DF_\theta^{-1}(F_\theta(x))\,g(F_\theta(x))-f(x)\big|^2=|f_\theta(x)-f(x)|^2 ,
\end{equation}
each costing one Jacobian--vector product by forward-mode automatic differentiation \citep{chen2024autodiff}. They are not equivalent for optimization. If the reference is invariant under a scaling $y\mapsto Dy$ in some directions (the transverse directions of the translation, all directions of a linear flow, the action of an action--angle flow with $\omega$ re-parametrized), then replacing $F_\theta$ by $DF_\theta$ leaves \eqref{eq:lpush} unchanged but multiplies the corresponding components of the pull-back residual by $D$: the pull-back form rewards shrinking the latent coordinates and makes $F_\theta^{-1}$ ill-conditioned. We use \eqref{eq:lpush}, which measures the learned vector field directly against $f$. In regime (3), $f$ is replaced by $f(\cdot;\alpha)$ and $\alpha$ is trained jointly.

\paragraph{Relative residual and scale-free sampling (R1).}
Near an equilibrium $|f(x)|\to0$, so the squared residual \eqref{eq:lpush} on uniformly sampled points carries almost no weight exactly where the invariants of the equilibrium act, and a mismatch of the spectrum can be hidden in a small ball at negligible cost (Section~\ref{sec:failures}). Two changes remove the problem together. The residual is made relative, $|f_\theta(x)-f(x)|^2/(|f(x)|^2+\varepsilon)$ with $\varepsilon=10^{-6}$, so that a given relative error costs the same at every scale; and half of each collocation batch is drawn with a log-uniform radius $|x|\in[10^{-3},1]$ and a uniform direction, so that every decade of $|x|$ around the equilibrium carries the same number of points. Together with the anchors $F_\theta(0)=0$, $DF_\theta(0)=I$, this makes the spectrum of $A$ observable from the residual at every scale: near $0$ the residual is $|(A-Df(0))x|^2/|Df(0)x|^2$, a direct scale-free gradient signal on $A$. Neither change alone sufficed in Experiment~7; both together recovered the spectrum in three and in six dimensions. On a field without an equilibrium in the domain ($|f|$ bounded away from zero) the relative residual is a mild reweighting and did not help (Experiment~1).

\paragraph{Data loss.}
If $F$ conjugates $\Phi$ to $\Psi$ then $F(x^{(k)})=\Psi_{t^{(k)}-t^{(j)}}(F(x^{(j)}))$ for any two observations on one trajectory. The latent residual $|F_\theta(x^{(k)})-\Psi_{\Delta t}F_\theta(x^{(j)})|^2$ has the scaling degeneracy of the pull-back form (it was tried and collapsed the interior of the limit cycle, Section~\ref{sec:failures}), so we measure the residual in the original space,
\begin{equation}\label{eq:ldat}
\ell^{\mathrm{dat}}(x^{(j)},x^{(k)},\Delta t)=|x^{(k)}-\Phi^\theta_{\Delta t}(x^{(j)})|^2 ,
\end{equation}
on pairs $(j,k)$ drawn at random with lags up to a maximal $\Delta t$. Pairs far apart in time cost the same as consecutive ones. In practice the maximal lag is started small and doubled whenever the running mean of the loss falls below a threshold tied to the noise floor; lengthening it on a schedule destroyed fits that were already poor. Where the data are unevenly distributed over the phase portrait, pairs are drawn stratified (equal shares from the neighborhood of the equilibrium, from the transient and from the attractor) so that every invariant is seen by some pairs.

\paragraph{Anchors.}
The anchors of Section~\ref{sec:anchor} are imposed as penalties: $|F_\theta(c,\xi)-(0,\xi)|^2$ on sampled $\xi$ for the translation, $|F_\theta(x^*)|^2+\|DF_\theta(x^*)-I\|^2$ for a known equilibrium. A point anchor $F_\theta(x^*)=0$ alone pins the equilibrium, which suffices to cross a stability barrier, but it does not pin the spectrum (Section~\ref{sec:failures}).

\paragraph{Symplectic transformation with gradient-form frequency map (R2).}
For a Hamiltonian target with $d$ degrees of freedom the action--angle reference with $\omega=\omega_0+\nabla K$ (Lemma~\ref{lem:gradient}) and a symplectic $F_\theta$ give a learned flow that is Hamiltonian for every parameter value (Proposition~\ref{prop:flow}(iv)) and whose latent actions $\tfrac12(Q_i^2+P_i^2)$ are, when the conjugacy is exact, the true action integrals of the invariant tori, since a symplectic conjugacy preserves the action integrals $\oint p\,\dd q$. A non-symplectic $F_\theta$ with the same reference conserves $|F_\theta|^2$ but may distort the actions while reproducing the frequencies, which are conjugacy invariants (Experiment~9).

\paragraph{Initialization of the invariants (R3).}
The parameters $\beta$ of the reference are invariants of the target and can usually be estimated before training. When $f$ is known and $x^*$ is a known equilibrium, $A$ or $\omega_0$ is initialized from $Df(x^*)$, as \citet{bizzi2025ncf} also do; at a minimum the sense of rotation is read off $Df(x^*)$. When only data are available, the sense and rate of rotation are the mean signed increment of the unwrapped polar angle along the trajectories, the Floquet exponent of an attracting cycle is the mean slope of the logarithm of the distance to the cycle during the approach, and the eigenvalues of the focus are those of a linear field fitted by least squares to finite differences of the data points nearest the equilibrium. In the Van der Pol experiment these estimates were within $\expFourEEstErrMin$--$\expFourEEstErrMax\%$ of the true invariants, and the oracle run of Section~\ref{sec:failures} suggests that such estimates should be kept fixed rather than released, since training moved them away from the truth.

\paragraph{Training.}
The total loss is a weighted sum of the active terms, each a Monte-Carlo average over fresh samples per step. We use Adam with cosine decay followed by L-BFGS on a fixed batch, in double precision. After training, $\Phi^\theta_t(x)$ for any $t$ costs one evaluation of $F_\theta$, $\Psi_t$ and $F_\theta^{-1}$.

\subsection{Diagnostics}\label{sec:diagnostics}
The training loss is not a reliable measure of success: in Experiments~7 and~8 models whose trajectory errors differ by an order of magnitude have plain equation losses on uniform points that agree to within $20\%$. Three diagnostics separate a conjugacy from an interpolant.
\begin{enumerate}[nosep,leftmargin=*]
\item \emph{Relative residual on spheres.} $|f_\theta-f|/|f|$ averaged over spheres $|x|=r$ for $r$ from $10^{-3}$ to the size of the domain must be flat in $r$. A residual that grows as $r\to0$ means the invariants of the equilibrium are wrong and the mismatch is hidden where $|f|$ is small.
\item \emph{Jacobian deviation.} With $DF_\theta(0)=I$ anchored, $\|DF_\theta(x)-I\|$ averaged over the sphere of radius $r$ must grow linearly in $r$, as for any smooth map with that Jacobian. A deviation that jumps to order one within $r<10^{-2}$ means the anchor is satisfied at the point and circumvented next to it.
\item \emph{Latent view.} The \emph{true} trajectories mapped by $F_\theta$ must be orbits of the reference (straight lines, circles, spirals of $e^{At}$, hyperbolae). A model whose latent trajectories are not reference orbits is a local interpolant that the exact group law extrapolates to nonsense beyond the lags it was trained on, whatever its loss, and must not be accepted.
\end{enumerate}
The first two are cheap and were the instruments that diagnosed the failures of Section~\ref{sec:failures}; the third is the acceptance test we apply to every model.

\subsection{Error propagation}
\begin{proposition}[Time-uniform error bound]\label{prop:error}
Let $F$ be an exact transformation on an invariant set $K$, $\Phi_t=F^{-1}\Psi_tF$, and let $F_\theta$ satisfy $\sup_K|F_\theta-F|\le\varepsilon$ and $\sup_{\Psi_{[0,T]}(F(K))}|F_\theta^{-1}-F^{-1}|\le\varepsilon'$. Then for $x\in K$ and $t\in[0,T]$,
\begin{equation}
|\Phi^\theta_t(x)-\Phi_t(x)|\le\varepsilon'+\Lip(F^{-1})\,\Lip(\Psi_t|_{F(K)})\,\varepsilon .
\end{equation}
In particular, for the translation reference, for a non-expanding reference on a compact set, and for the action--angle reference with bounded $\omega'$ up to a factor linear in $t$, the error is bounded uniformly in $t$.
\end{proposition}
\begin{proof}
$|F_\theta^{-1}(\Psi_tF_\theta x)-F^{-1}(\Psi_tFx)|\le|F_\theta^{-1}(\Psi_tF_\theta x)-F^{-1}(\Psi_tF_\theta x)|+|F^{-1}(\Psi_tF_\theta x)-F^{-1}(\Psi_tFx)|$, and the second term is at most $\Lip(F^{-1})\Lip(\Psi_t)|F_\theta x-Fx|$. For the translation $\Lip(\Psi_t)=1$; for the action--angle flow $D\Psi_t$ is a rotation plus a shear of size $|\omega'|\,|y|\,t$.
\end{proof}
The learned flow does not accumulate error because it does not step: the only amplification is the intrinsic one of the reference, which by Proposition~\ref{prop:invariants} is also that of the target. The bound is conditional on a small \emph{uniform} error of $F_\theta$, and the experiments show that this is where the difficulty lies.

\section{Experiments}\label{sec:experiments}
All experiments use one code base (PyTorch, double precision, CPU). $F_\theta$ has six to twelve affine coupling blocks with two hidden layers of width $64$ (eight gradient modules of width $64$ in the symplectic case). Training is Adam with cosine decay ($6\,000$--$15\,000$ steps, fresh collocation batches of $2\,048$ points) followed by L-BFGS on a fixed batch. Reference solutions use an eighth-order Runge--Kutta method at tolerance $10^{-11}$. We report the mean trajectory error over the test initial conditions and horizon. The group-law violation $\max_x|\Phi^\theta_s\Phi^\theta_tx-\Phi^\theta_{s+t}x|$ and the inverse consistency $\max_x|F_\theta^{-1}F_\theta x-x|$ are at rounding level ($10^{-15}$ to $10^{-13}$) in every run, as they must be by construction, and are not repeated. In the figures, black lines are reference solutions and red dashed lines the learned flow; the ``latent'' panels show $F_\theta$ applied to true trajectories. Table~\ref{tab:overview} lists the experiments and their outcome; this section presents the final configuration of each, Section~\ref{sec:failures} the runs that led to it. The complete chronology of every run, including those not shown here, is in the experiment report distributed with the code (Appendix~\ref{app:report}).

\begin{table}[t]
\caption{Overview. Regime: (1) $f$ known, (2) time series only, (3) $f$ known up to $\alpha$. ``Seeds'' is the number of seeds the final configuration was run with.}
\label{tab:overview}
\centering\footnotesize
\begin{tabular}{llllp{3.3cm}cp{4.1cm}}
\toprule
\# & system & $n$ & regime & reference $\Psi$; learned in $\Psi$ & seeds & outcome \\
\midrule
1 & dispersive field & 2 & 1 & translation; -- & 1 & correct, error flat in $t$; time-input PINN more accurate inside its horizon \\
2 & nonlinear saddle & 2 & 1 & diagonal linear; eigenvalues & 1 & spectrum $\expTwoLamJ$ with Jacobian anchor \\
3 & pendulum (librations) & 2 & 1 & action--angle; $\omega(I)$ & 1 & frequency map recovered; $\omega(0)$ unanchored \\
4 & Van der Pol, data only & 2 & 2 & decoupled limit cycle; rates, Floquet exponent, frequencies, radius & 1 & open: conjugacy on the data region, Floquet exponent $\expFourGInvErrFloquet\%$ off \\
5 & Lotka--Volterra & 2 & 3 & action--angle; $\omega(I)$, $\alpha=(a,b,c,d)$ & 1 & parameters to $\expFiveAlphaErr$ \\
6 & Duffing (Hamiltonian) & 2 & 1 & action--angle; $\omega(I)$; symplectic $F_\theta$ & 1 & exactly symplectic; RealNVP: one of three seeds diverged \\
7 & 3-D nonlinear sink & 3 & 1 & $e^{At}$, full $A$; spectrum & 3 & spectrum to $\expSevenSeedsEigErrMax$ (R1) \\
8 & 6-D nonlinear sink & 6 & 1 & $e^{At}$, full $A$; spectrum & 3 & all six eigenvalues within $\expEightSeedsEigRelMax\%$ (R1) \\
9 & 4-D integrable Hamiltonian & 4 & 1 & action--angle, $\omega=\omega_0+\nabla K$; symplectic $F_\theta$ & 3 & frequencies and actions to $<1\%$, exactly symplectic (R2, R3) \\
\bottomrule
\end{tabular}
\end{table}

\subsection{Sinks in three and six dimensions: recovering the spectrum (R1)}\label{sec:sinks}
\emph{Three dimensions.} The system $\dot x_1=-\tfrac12x_1-2x_2+x_3^2$, $\dot x_2=2x_1-\tfrac12x_2+x_1x_3$, $\dot x_3=-x_3+\tfrac12x_1x_2$ has a sink at the origin with spectrum $\{-\tfrac12\pm2i,-1\}$. The reference is $e^{At}$ with a full learnable $A\in\R^{3\times3}$ initialized near $-0.3I$, trained with R1 on $[-1,1]^3$ ($8\,000$ Adam steps, $500$ L-BFGS iterations). The spectrum converged to $\{\expSevenEigR\}$, a maximal deviation of $\expSevenEigErrR$ from the true one; the imaginary part passed $1.9$ at iteration $\expSevenItImR$ and the final approach happened during L-BFGS (Figure~\ref{fig:exp7b}(e)). The equilibrium is at $\expSevenFpR$ from the origin. The diagnostics of Section~\ref{sec:diagnostics} pass: the relative residual on spheres is $\expSevenRelTinyR$ at $|x|=10^{-3}$, $\expSevenRelSmallR$ at $10^{-2}$ and $\expSevenRelMidR$ at $10^{-1}$, flat across the decades, and $\|DF_\theta(x)-I\|$ grows linearly from $\expSevenJacDevTinyR$ at $10^{-3}$ through $\expSevenJacDevSmallR$ at $10^{-2}$ to $\expSevenJacDevMidR$ at $10^{-1}$ (panel (f)). The trajectory error over $t\in[0,8]$ from the sphere of radius $0.8$ is $\expSevenErrR$ (maximum $\expSevenErrMaxR$, relative $L^2$ error $\expSevenRelLR$) and decreases in $t$ as trajectories enter the region where the fit is best; the latent view shows the true trajectories as logarithmic spirals of $e^{At}$ (panel (c)). Training took $\expSevenTimeR$ minutes.
\begin{figure}[t]
\centering\includegraphics[width=0.88\linewidth]{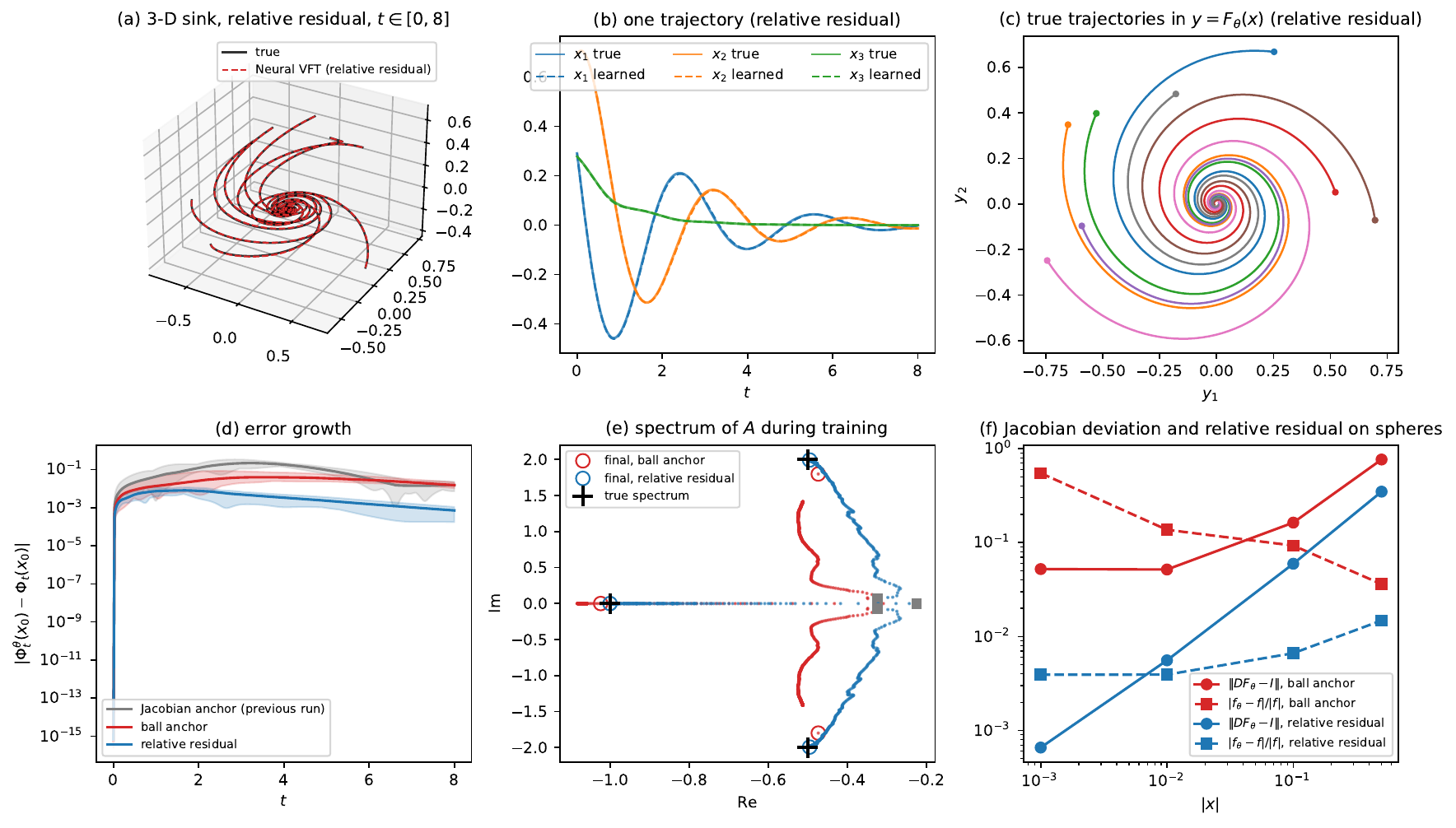}
\caption{Three-dimensional sink, R1. (a) Eight trajectories from the sphere of radius $0.8$. (b) One trajectory, componentwise. (c) True trajectories in the learned coordinates: spirals of $e^{At}$. (d) Error growth for the earlier Jacobian-anchored run (Section~\ref{sec:failures}) and the two runs of the final comparison (ball anchor, relative residual), mean and range over eight trajectories. (e) Spectrum of $A$ during training (gray squares: initial, crosses: true, circles: final). (f) $\|DF_\theta(x)-I\|$ and the relative residual $|f_\theta-f|/|f|$ averaged over spheres of radius $|x|$.}
\label{fig:exp7b}
\end{figure}

\emph{Six dimensions.} The system $\dot x=Ax+\tfrac12(x_3^2,\;x_1x_5,\;x_4x_6,\;x_2^2,\;x_1x_2,\;x_3x_5)$ with block-diagonal $A$ has spectrum $\{\expEightTrueEig\}$: two complex pairs with different damping and rotation rates and two real eigenvalues an order of magnitude apart. The reference is $e^{At}$ with a full learnable $A\in\R^{6\times6}$ ($36$ parameters, of which six are invariants and $30$ centralizer freedom removed by the Jacobian anchor), the transformation a twelve-block coupling network with $\expEightParamsR$ parameters, trained with R1 for $6\,000$ Adam steps and $500$ L-BFGS iterations on $[-1,1]^6$. The learned spectrum is $\{\expEightEigR\}$: every eigenvalue within $\expEightEigErrRelR\%$ (largest absolute deviation $\expEightEigErrR$, on the eigenvalue $-2$), with both complex pairs and the slow real direction resolved. The equilibrium offset is $\expEightFpR$, the relative residual is flat at $\expEightRelTinyR$--$\expEightRelMidR$ from $|x|=10^{-3}$ to $10^{-1}$, the Jacobian deviation grows from $\expEightJacDevTinyR$ at $10^{-3}$ to $\expEightJacDevMidR$ at $10^{-1}$ (Figure~\ref{fig:exp8}(f)), and the eight test trajectories over $t\in[0,8]$ have mean error $\expEightErrR$ (maximum $\expEightErrMaxR$, relative $L^2$ error $\expEightRelLR$). Training took $\expEightTimeR$ minutes. Doubling the dimension therefore changed neither the recipe nor the budget.
\begin{figure}[t]
\centering\includegraphics[width=0.88\linewidth]{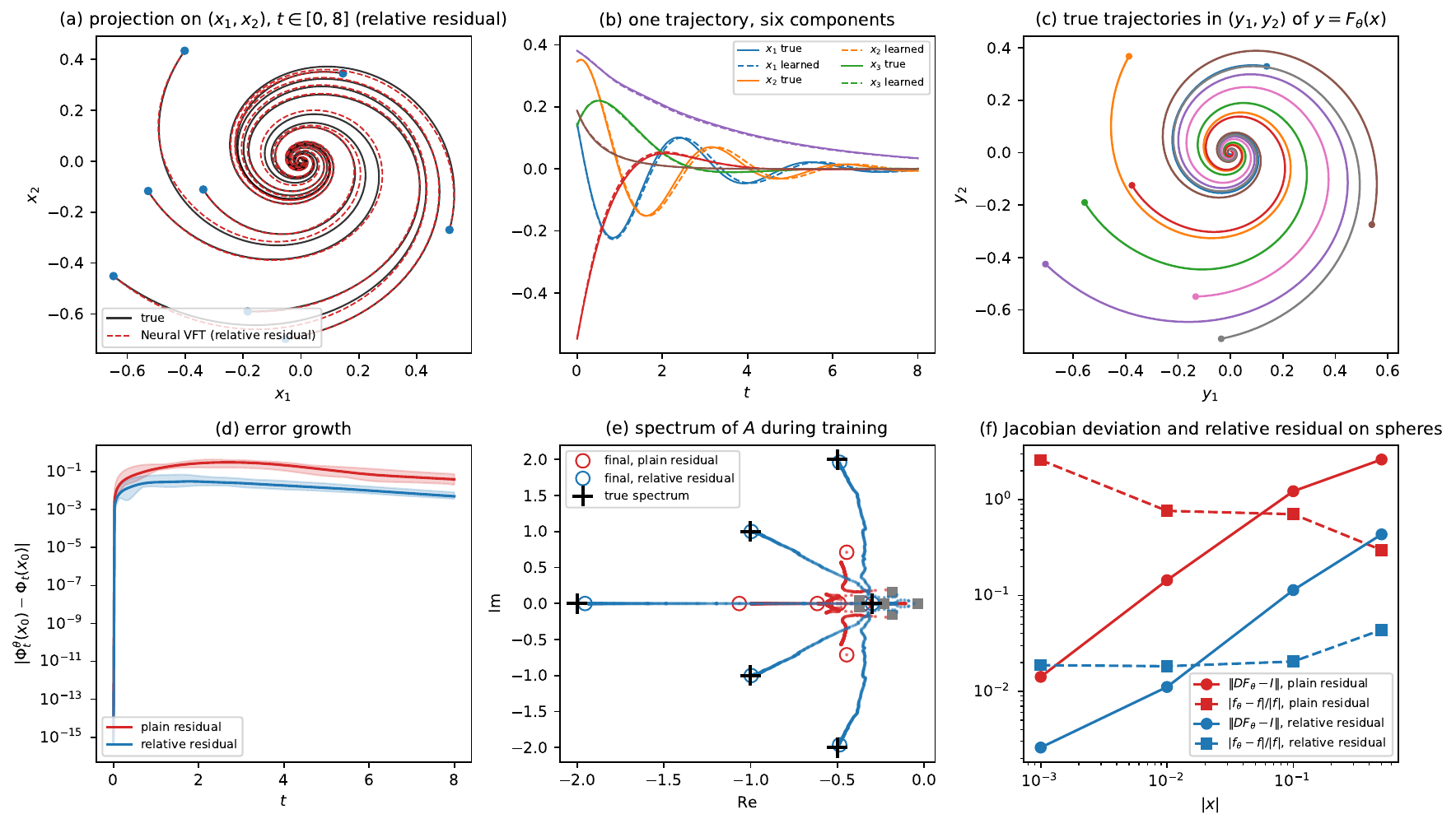}
\caption{Six-dimensional sink, R1 against the plain residual. (a) Eight trajectories from the sphere of radius $0.8$, projected on $(x_1,x_2)$. (b) One trajectory, six components. (c) True trajectories in the first two learned coordinates. (d) Error growth. (e) Spectrum of $A$ during training of both runs (gray squares: initial, crosses: true, circles: final). (f) $\|DF_\theta(x)-I\|$ and the relative residual on spheres of radius $|x|$.}
\label{fig:exp8}
\end{figure}

\subsection{Hamiltonian systems: symplectic transformation and action--angle variables (R2, R3)}\label{sec:hamiltonian}
\emph{Duffing oscillator.} For $H=\tfrac12p^2+\tfrac12q^2+\tfrac14q^4$ we compare a symplectic $F_\theta$ with a RealNVP $F_\theta$, both with the action--angle reference and the equation loss on $H<3$. By Proposition~\ref{prop:flow}(iii) both learned flows conserve $|F_\theta|^2$ exactly; by (iv) only the symplectic one is Hamiltonian. Over $t\in[0,100]$ (up to $16$ periods) the symplectic model has trajectory error $\expSixErrS$, growing slowly through a phase drift, maximal energy error $\expSixEnergyS$ with no secular trend, and symplecticity defect $\|D\Phi_t^{\theta\top}JD\Phi_t^\theta-J\|=\expSixSympS$ at $t=50$ over $500$ points (Figure~\ref{fig:exp6}). The RealNVP run with the same seed diverged (trajectory error $\expSixErrR$); two further seeds converged (\expSixSeeds), with long-horizon errors larger than the symplectic model's. Section~\ref{sec:failures} traces the divergent seed to the orientation barrier.
\begin{figure}[t]
\centering\includegraphics[width=0.72\linewidth]{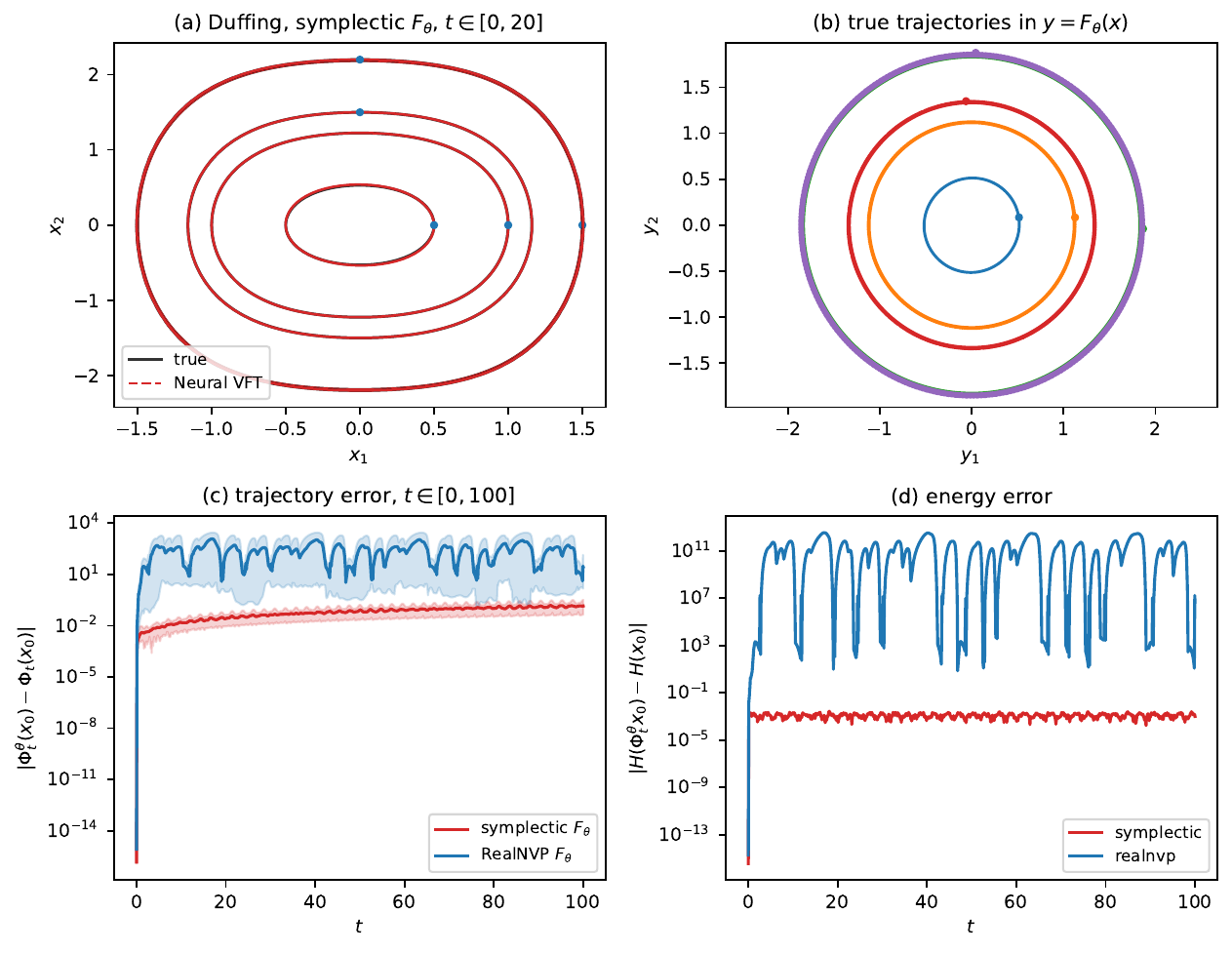}
\caption{Duffing oscillator. (a) Orbits with the symplectic transformation. (b) True trajectories in its learned coordinates. (c) Trajectory error over $t\in[0,100]$ for the two architectures (the RealNVP run shown is the divergent seed). (d) Energy error.}
\label{fig:exp6}
\end{figure}

\emph{Four-dimensional integrable system.} The system is a pair of Duffing oscillators, $H_0=\sum_{i=1}^2\tfrac12p_i^2+\tfrac12q_i^2+\tfrac14q_i^4$, written in the coordinates $z'=G(z)$ of a fixed nonlinear symplectic map $G=S_p\circ S_q$ with $S_q\colon q\mapsto q+\nabla_pW(p)$, $W=\tfrac3{20}(p_1p_2)^2$, and $S_p\colon p\mapsto p+\nabla_qU(q)$, $U=\tfrac3{20}(q_1q_2)^2$. The field in $z'$ couples all four components, but its invariants are those of two uncoupled Duffing oscillators: the frequency maps $\omega_i(I_i)$, known by quadrature, and the actions themselves. The reference is the action--angle flow with $\omega=\omega_0+\nabla K(I)$, $\omega_0=-1$ read off the clockwise rotation of the field (R3), and two transformations are compared: a symplectic network with eight gradient layers ($\expNineParamsS$ parameters) and a RealNVP with eight coupling blocks ($\expNineParamsR$), both trained with the equation loss on $2\,048$ points per step in $[-1.6,1.6]^4$ with $H<1.5$, $6\,000$ Adam steps and $500$ L-BFGS iterations. \emph{Symplectic $F_\theta$:} equation loss $\expNineLossS$, vector-field error $\expNineVfS$, trajectory error $\expNineErrShortS$ over $t\in[0,20]$ and $\expNineErrS$ over $t\in[0,50]$ (about eight periods; maximum $\expNineErrMaxS$), energy error at most $\expNineEnergyS$, symplecticity error $\expNineSympS$ of the time-$50$ map. The learned frequency maps coincide with the true Duffing map (Figure~\ref{fig:exp9}(f)): median relative error $\expNineFreqMedAS\%$ and $\expNineFreqMedBS\%$ over $500$ collocation points (mean $\expNineFreqErrAS\%$ and $\expNineFreqErrBS\%$, maximum $\expNineFreqErrMaxS\%$ near $I=0$), and the latent actions $\tfrac12(Q_i^2+P_i^2)$ of $F_\theta(x)$ agree with the true action integrals of the tori through $x$ to $\expNineActMedAS\%$ and $\expNineActMedBS\%$ (median over points with $I>0.1$): the symplectic conjugacy recovers the action--angle variables themselves. \emph{RealNVP $F_\theta$:} equation loss $\expNineLossR$, trajectory error $\expNineErrR$, energy error up to $\expNineEnergyR$, frequencies again correct to $\expNineFreqMedAR\%$ and $\expNineFreqMedBR\%$, but the time-$50$ map is not symplectic (error $\expNineSympR$) and the latent actions are $\expNineActMedAR\%$ and $\expNineActMedBR\%$ off, as Section~\ref{sec:losses} anticipates. Training took $\expNineTimeS$ and $\expNineTimeR$ minutes.
\begin{figure}[t]
\centering\includegraphics[width=0.88\linewidth]{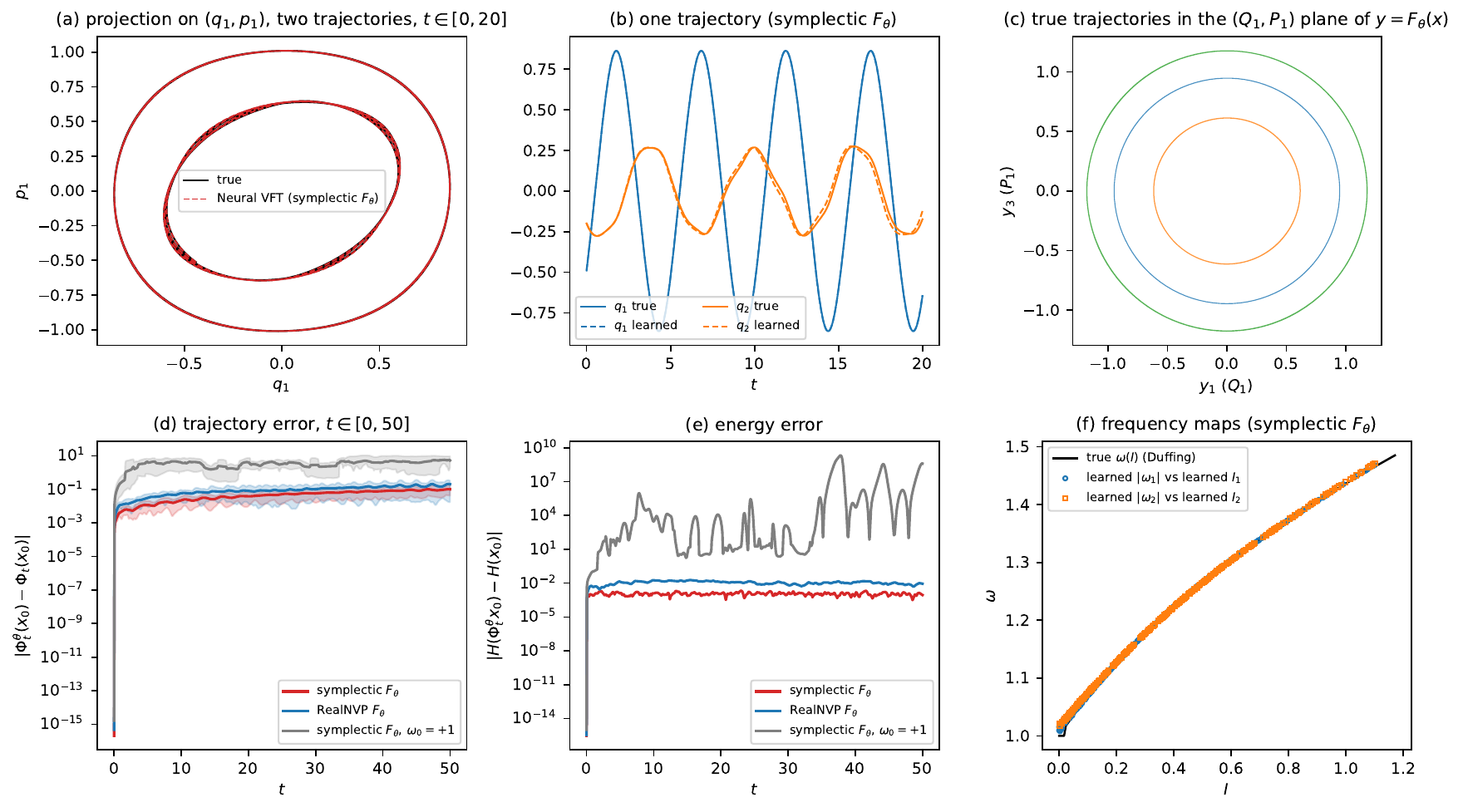}
\caption{Four-dimensional integrable Hamiltonian system. (a) Two trajectories projected on $(q_1,p_1)$, $t\in[0,20]$. (b) One trajectory, $q_1$ and $q_2$. (c) True trajectories in the $(Q_1,P_1)$ plane of the learned coordinates. (d) Trajectory error over $t\in[0,50]$ for both transformations and for the run with the wrong orientation (gray). (e) Energy error. (f) Learned frequencies against learned actions for the two degrees of freedom, with the true Duffing frequency map.}
\label{fig:exp9}
\end{figure}

\subsection{Seed robustness}\label{sec:seeds}
Experiment~6 showed that a single seed can mislead. The three runs that carry the method's claims, R1 on the three- and the six-dimensional sink and R2--R3 on the four-dimensional Hamiltonian system, were repeated with seeds $1$ and $2$ and everything else unchanged (seed changes the initialization of $F_\theta$ and of the reference, the collocation points and the L-BFGS batch). Table~\ref{tab:seeds} lists the per-seed numbers. On the six-dimensional sink all seeds recover the spectrum with a largest relative eigenvalue error of $\expEightSeedsEigRelMax\%$; the seed-$0$ run reported above is the least accurate of the three. On the Hamiltonian system the seeds agree to within $25\%$ on every metric, with frequency maps correct to a median of at most $\expNineSeedsFreqMax\%$ and actions to at most $\expNineSeedsActMax\%$, and no trace of the orientation barrier once $\omega_0$ is initialized from the field. On the three-dimensional sink the seeds are indistinguishable (learned spectra \expSevenSeedsSpectra{} against $\{-0.5\pm2i,-1\}$). These three results are the ones this paper claims as recipes.
\begin{table}[t]
\caption{Seed robustness; each cell lists seed 0 / seed 1 / seed 2. Spectrum deviation: largest distance between a learned eigenvalue and the nearest true one. Trajectory error: mean over the test trajectories and horizon of the experiment. Frequency and action errors: median relative error over $300$ collocation points, two degrees of freedom each.}
\label{tab:seeds}
\centering\scriptsize
\begin{tabular}{p{2.9cm}p{2.9cm}p{3.9cm}p{3.9cm}}
\toprule
 & 3-D sink, R1 & 6-D sink, R1 & 4-D Hamiltonian, symplectic $F_\theta$ \\
\midrule
spectrum deviation & $\expSevenSeedsEigErr$ & $\expEightSeedsEigErr$ & -- \\
largest relative eigenvalue error (\%) & -- & $\expEightSeedsEigRel$ & -- \\
equilibrium offset $|F_\theta^{-1}(0)|$ & $\expSevenSeedsFp$ & $\expEightSeedsFp$ & -- \\
trajectory error (mean) & $\expSevenSeedsErr$ & $\expEightSeedsErr$ & $\expNineSeedsErr$ \\
equation loss & -- & -- & $\expNineSeedsLoss$ \\
energy error (max) & -- & -- & $\expNineSeedsEnergy$ \\
symplecticity error, $t=50$ & -- & -- & $\expNineSeedsSymp$ \\
frequency error, median (\%) & -- & -- & $\expNineSeedsFreq$ \\
action error, median (\%) & -- & -- & $\expNineSeedsAct$ \\
training time (min) & $\expSevenSeedsTime$ & $\expEightSeedsTime$ & $\expNineSeedsTime$ \\
\bottomrule
\end{tabular}
\end{table}

\subsection{Planar experiments with a single seed}\label{sec:planar}
\emph{Pendulum: a learnable frequency--amplitude relation.} Librations of $\ddot q=-\sin q$ fill $E=\tfrac12p^2+1-\cos q<2$, a nonlinear center whose period grows without bound at the separatrix. With the action--angle reference and a learnable $\omega(I)$, trained with the equation loss on $E<1.8$, the trajectory error over $t\in[0,30]$ (about five periods at the largest amplitude $q_0=2.4$) is $\expThreeErrL$ and essentially flat in $t$; the fixed-$\omega$ ablation, i.e.\ the harmonic-oscillator reference that Proposition~\ref{prop:invariants}(c) rules out, has error $\expThreeErrFixed$ because the phases drift; a time-input baseline trained on $t\in[0,5]$ reaches $\expThreeErrBaseShort$ on its horizon and $\expThreeErrBase$ over $[0,30]$. The learned $\omega(I)$ reproduces $2\pi/T(q_0)$ over most of the libration region (Figure~\ref{fig:exp3}(c)); in the run shown it bends away from the correct limit at small amplitude (relative error $\expThreeFreqErr$ at the smallest amplitudes), while a retrained copy with the same settings, differing through thread-level non-determinism, did not (maximal relative frequency error $\expThreeFreqErrLarge$, at the edge of the training region; trajectory error $\expThreeBErr$). The run-dependent deviation at $I\to0$ is exactly what R1 addresses: $|f|$ vanishes at the equilibrium and nothing in the plain loss pins $\omega(0)$.
\begin{figure}[t]
\centering\includegraphics[width=0.72\linewidth]{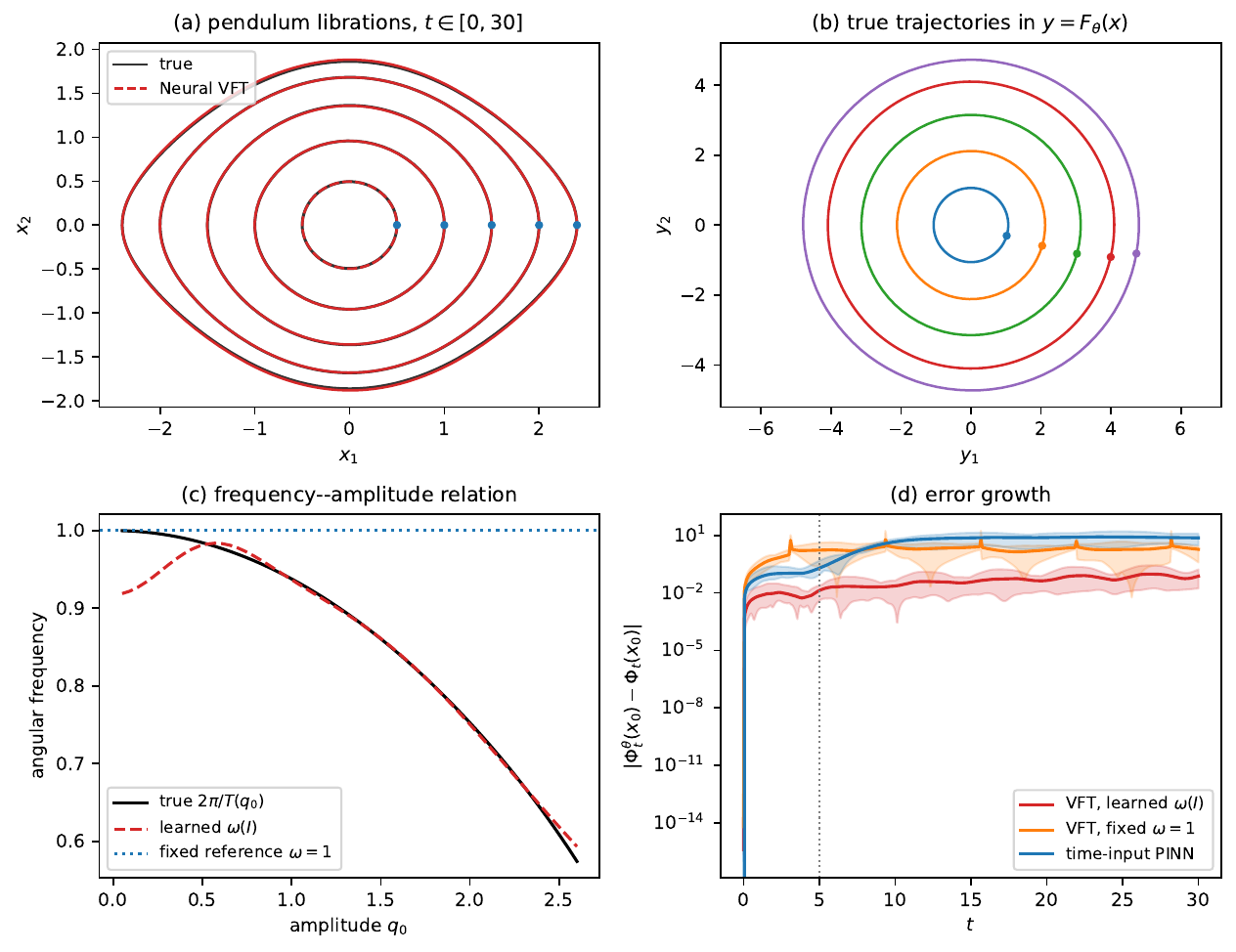}
\caption{Pendulum. (a) Librations with amplitudes $q_0\in\{0.5,1,1.5,2,2.4\}$ over $t\in[0,30]$. (b) True trajectories in the learned coordinates. (c) Learned $\omega(I)$ along $(q_0,0)$ against the exact frequency $2\pi/T(q_0)$. (d) Error growth for the three models.}
\label{fig:exp3}
\end{figure}

\emph{Lotka--Volterra: parameter identification from noisy data.} For $\dot x=ax-bxy$, $\dot y=-cy+dxy$ with true $(a,b,c,d)=(1,0.5,0.8,0.4)$ we observe $12$ noisy trajectories (noise standard deviation $\expFiveNoise$, length $15$, $\Delta t=0.1$) and treat all four parameters as unknown, initialized at $(0.5,1,0.5,1)$; the loss combines \eqref{eq:ldat} with the equation residual for $f(\cdot;\alpha)$ at collocation points near the data, and the reference is the action--angle flow. The identified parameters are $(\expFiveAlpha)$, a maximal relative error of $\expFiveAlphaErr$; the data loss reached the noise floor of the pairwise residual; the learned equilibrium $F_\theta^{-1}(0)=\expFiveFp$ recovers $(2,2)$; the trajectory error is $\expFiveErrShort$ on the data horizon and $\expFiveErr$ over $t\in[0,40]$ (Figure~\ref{fig:exp5}), the only visible deviation being on the outermost test orbit, part of which lies outside the data. In this regime the method behaves as a physics-informed network with an exact flow structure.
\begin{figure}[t]
\centering\includegraphics[width=0.9\linewidth]{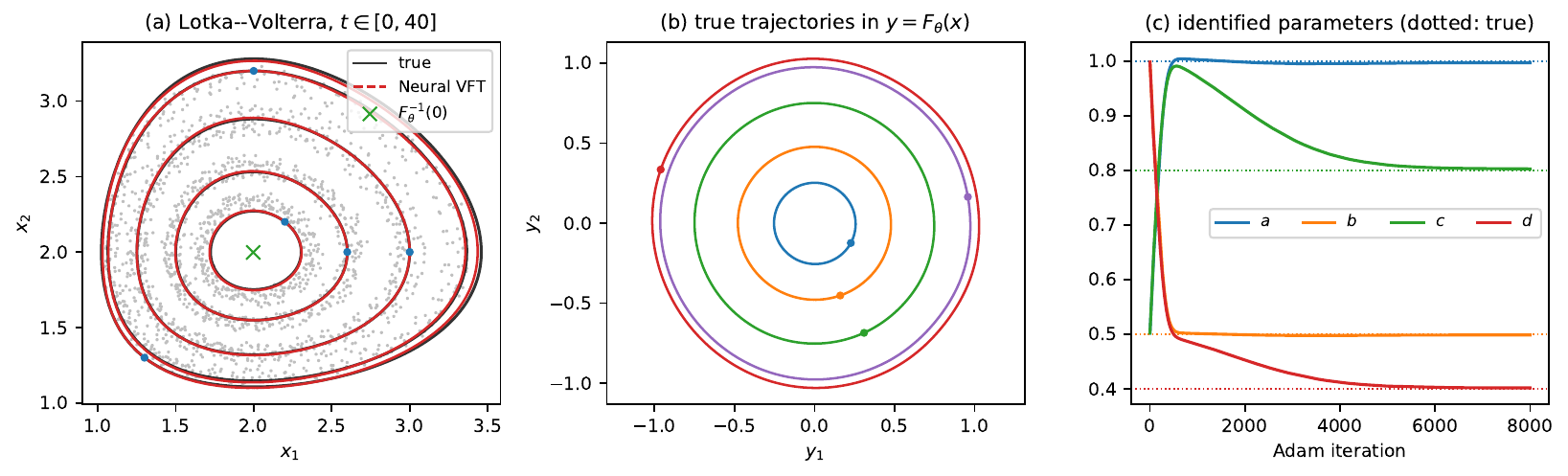}
\caption{Lotka--Volterra. (a) Test trajectories over $t\in[0,40]$ with the noisy training data (gray) and the learned equilibrium. (b) True trajectories in the learned coordinates. (c) Convergence of the four identified parameters.}
\label{fig:exp5}
\end{figure}

\emph{A nonlinear saddle: learning the stability type.} The saddle $\dot x_1=x_1+x_2^2$, $\dot x_2=-x_2$ is trained on $[-1.5,1.5]^2$ with the diagonal linear reference $\lambda$ initialized at $(-0.5,-0.5)$, a sink, so that the stability type has to be discovered. With the anchors $F_\theta(0)=0$, $DF_\theta(0)=I$ the eigenvalues converge to $\expTwoLamJ$, the equilibrium stays at $\expTwoFpJ$, the equation loss reaches $\expTwoLossJ$, the vector-field error $\expTwoVfJ$ and the trajectory error $\expTwoErrJ$ (Figure~\ref{fig:exp2}(a,b)). The two runs without the Jacobian anchor, also in Figure~\ref{fig:exp2}, failed and are the first two entries of Section~\ref{sec:failures}.
\begin{figure}[t]
\centering\includegraphics[width=0.72\linewidth]{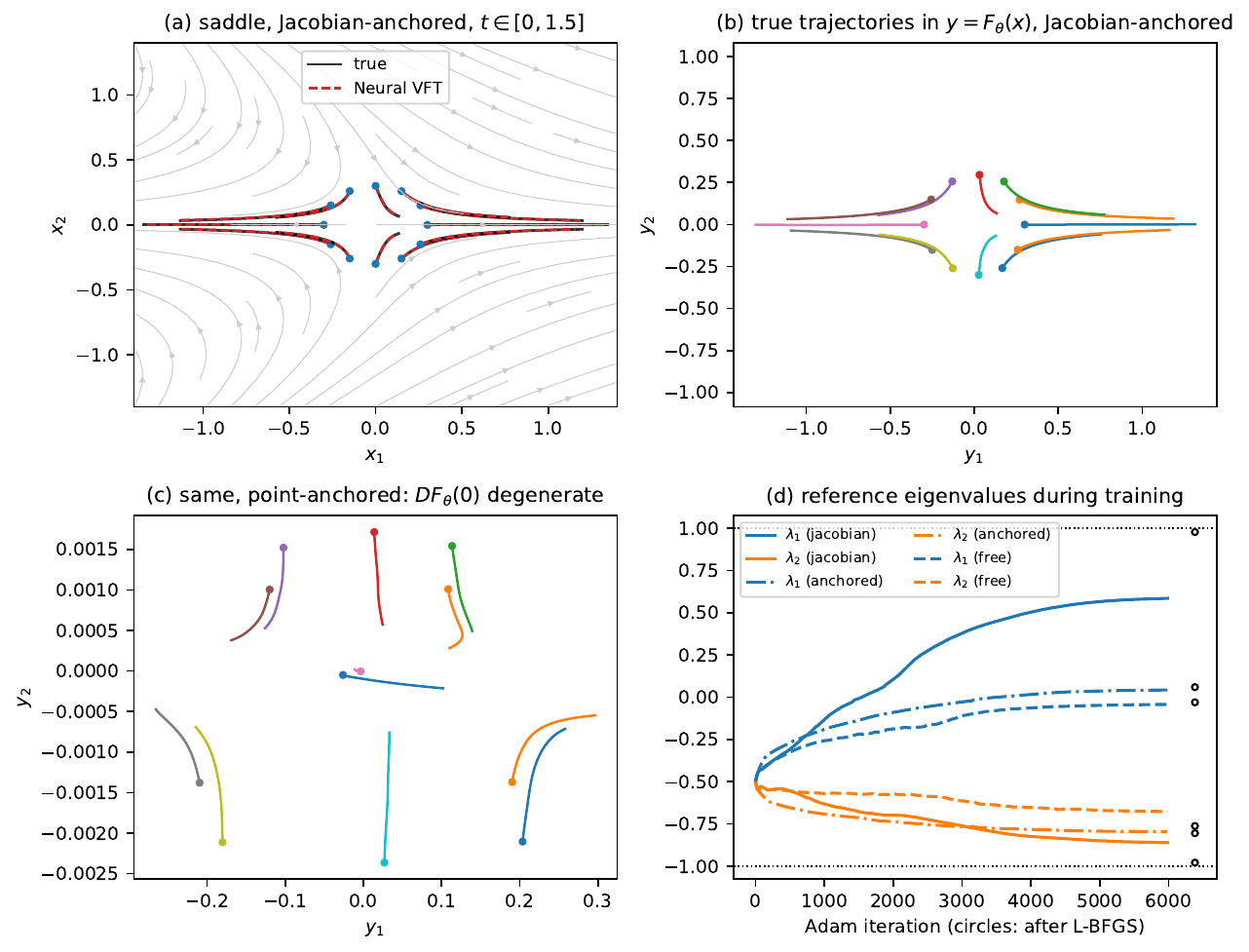}
\caption{Nonlinear saddle. (a) True and learned (Jacobian-anchored) trajectories over the stream plot of $f$. (b) True trajectories in the learned coordinates of the Jacobian-anchored model (hyperbolae). (c) The same for the point-anchored model: the transverse coordinate has collapsed. (d) Reference eigenvalues during training for the free, point-anchored and Jacobian-anchored runs; dotted lines mark $\pm1$, circles the values after L-BFGS.}
\label{fig:exp2}
\end{figure}

\subsection{Where the method is weaker}\label{sec:weaker}
\emph{A dispersive field against a time-input network.} The system $\dot x_1=1+\tfrac12\sin x_2$, $\dot x_2=\tfrac45\cos x_1+\tfrac15\sin x_2$ has $\dot x_1\geq\tfrac12$ and is globally trivializable (Proposition~\ref{prop:global}). We train the translation reference with the equation loss on $[-3,3]\times[-4.5,4.5]$ and the section anchor on $\{x_1=0\}$, and compare with a time-input network $\Phi_\theta(t,x)=x+t\,\mathrm{NN}(t,x)$ trained with the PINN residual on $t\in[0,1.5]$ with the same budget. The best transformation (twelve blocks, scale bound $1$, $\expOneDeepParams$ parameters, $15\,000$ Adam steps and $1\,000$ L-BFGS iterations) reaches an equation loss of $\expOneDeepEqLoss$ and a trajectory error of $\expOneDeepErr$ over $t\in[0,2.5]$ (maximum $\expOneDeepMax$), flat in $t$ as Proposition~\ref{prop:error} predicts, with a group-law violation of $\expOneGroupVft$ against $\expOneGroupBase$ for the baseline (Figure~\ref{fig:exp1b}). The baseline, however, reaches $\expOneBaseInRangeErr$ inside its horizon, an order of magnitude better, and $\expOneBaseErr$ over the full horizon, where its error grows; it trained in $\expOneTimeBase$ minutes against $\expOneDeepTime$. On this easy problem the conjugacy representation is harder to optimize than the direct one: the loss involves the Jacobian of a deep invertible map and the fit is an $H^1$-type fit of a time function and a first integral whose gradients vary by an order of magnitude across the box. An eight-block, width-$96$ transformation with $\expOneRelParams$ parameters plateaued at $\expOneVftErr$ (plain) and $\expOneRelErr$ (relative residual), so the gain came from depth with a tighter per-block scale, not from parameters or from R1.
\begin{figure}[t]
\centering\includegraphics[width=0.9\linewidth]{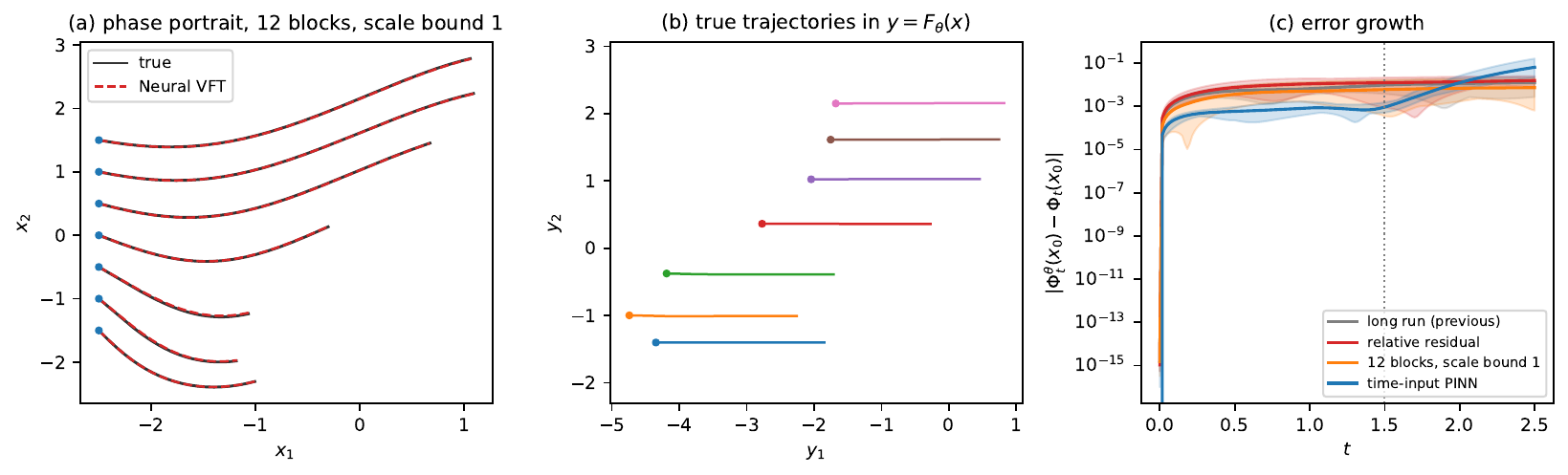}
\caption{Dispersive field. (a) True and learned trajectories of the twelve-block model. (b) Its latent view: straight lines traversed at unit speed. (c) Mean (line) and range (band) of the trajectory error for the eight-block long run, the relative-residual variant, the twelve-block model and the time-input baseline; the dotted line marks the baseline's training horizon.}
\label{fig:exp1b}
\end{figure}

\emph{Van der Pol from data only.} The oscillator $\ddot x=\mu(1-x^2)\dot x-x$ with $\mu=1.5$ has an unstable focus with eigenvalues $\expFourTrueRate\pm\expFourTrueOmegaO\,i$ and an attracting cycle of period $\expFourTruePeriod$ (frequency $\expFourTrueOmegaC$) with Floquet exponent $\expFourTrueFloquet$. The field is not used: $40$ trajectories of length $12$ sampled at $\Delta t=0.05$ from $[-2.5,2.5]^2$ with additive noise of standard deviation $0.005$ are the only input (noise floor of the loss $5\times10^{-5}$). The final configuration, reached after six earlier runs (Section~\ref{sec:failures}), uses the decoupled reference of Section~\ref{sec:reference} with its four invariants initialized from the data as in R3 (estimates $(\expFourEEstInv)$, within $\expFourEEstErrMin$--$\expFourEEstErrMax\%$ of the true $(\expFourBTrueInv)$, ordered as origin rate, Floquet exponent, origin frequency, cycle frequency) and frozen for the first $3\,000$ steps, a twelve-block scale-bound-$1$ transformation, the point anchor $F_\theta(0)=0$, pairs at most $0.5$ time units apart with the lag doubled when the running mean of the loss falls below ten times the noise floor, and stratified pairs: of the $9\,600$ starting points only $\expFourStrataNear\%$ lie within $|x|<1$ of the focus and $\expFourStrataTrans\%$ in transit, the rest on the cycle, so each batch takes a quarter from each of the first two strata. The outcome: the invariants ended at $(\expFourGInv)$, the origin rate within $\expFourGInvErrRate\%$ and the cycle frequency within $\expFourGInvErrWcDec\%$, the origin frequency $\expFourGInvErrWo\%$ and the Floquet exponent $\expFourGInvErrFloquet\%$ off. On fixed batches under the natural pair distribution the loss is $\expFourGLossOneStep$ for single steps, $\expFourGLossLagHalf$, $\expFourGLossLagOne$, $\expFourGLossLagTwo$ and $\expFourGLossLagFour$ for lags up to $0.5$, $1$, $2$ and $4$ (the last two never trained on). The five test trajectories inside the sampled box have mean error $\expFourGErrInBox$ over $t\in[0,20]$, almost three periods (final $\expFourGErrInBoxFinal$), the learned field has relative $L^2$ error $\expFourGVfInner$ on $|x|<2$ (median $\expFourGVfMedian$), and the latent view shows true trajectories spiraling onto a circle (Figure~\ref{fig:exp4f}). A sixth trajectory started outside the box diverges and the empty corners of the box are wrong (overall $\expFourGErr$ and $\expFourGVf$), as they must be without data. We call this a conjugacy on the data region, not a solved identification: the Floquet exponent is not recovered, and Section~\ref{sec:failures} explains why we attribute this to the transformation rather than to the loss.
\begin{figure}[t]
\centering\includegraphics[width=0.88\linewidth]{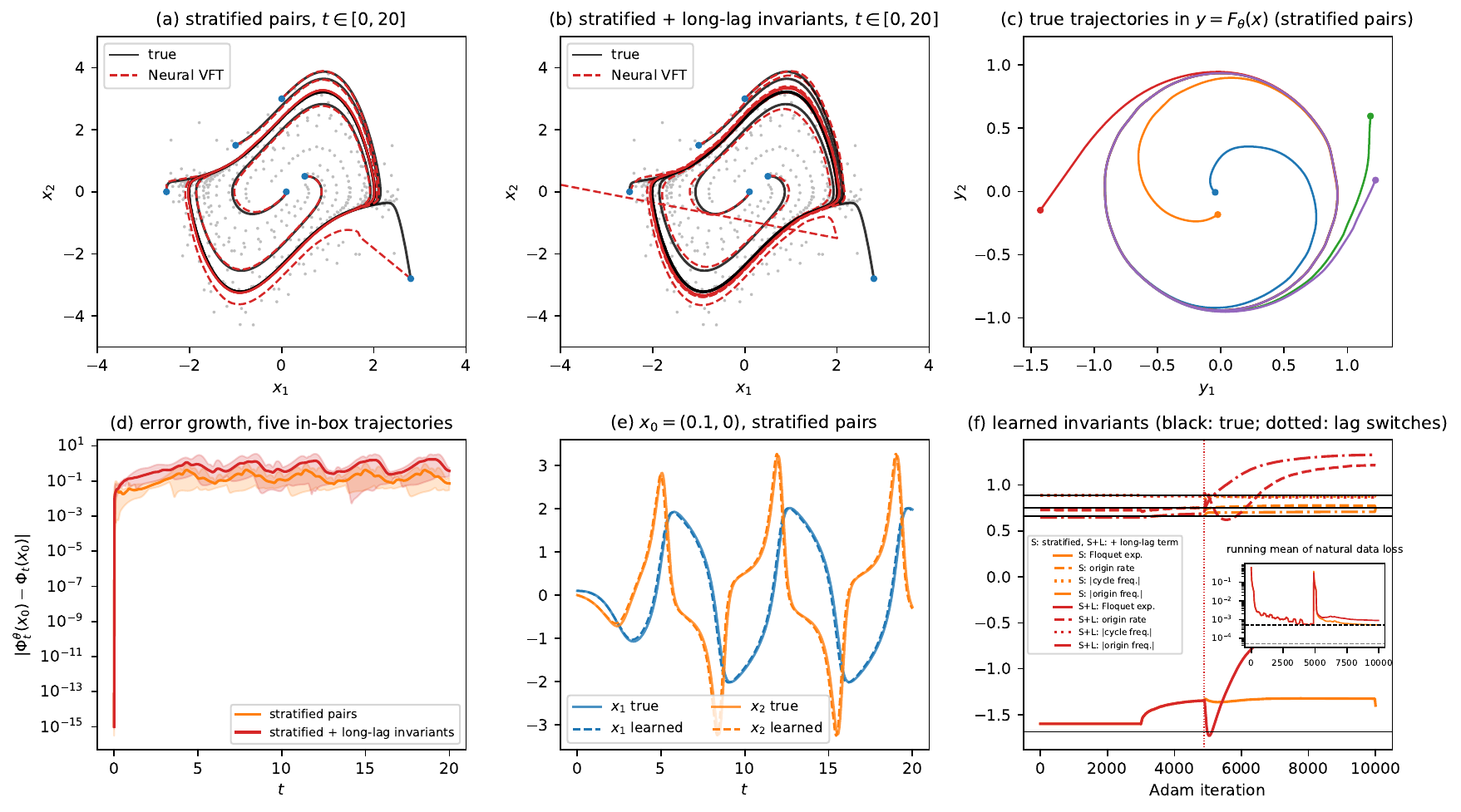}
\caption{Van der Pol from data, final configuration (stratified pairs, invariants initialized from the data) and, in (a,b), the run with an additional long-lag term that failed. (a),(b) Learned flows from six test initial conditions (gray dots: training data; axes clipped); the trajectory from $(2.8,-2.8)$ starts outside the sampled box. (c) True in-box trajectories in the learned coordinates: spirals onto a circle. (d) Error growth over the five in-box trajectories. (e) Time series from $x_0=(0.1,0)$. (f) Learned invariants during training (black lines: true values; dotted line: lag doubling); inset: running mean of the loss under the natural pair distribution with the criterion (black dashed) and the noise floor (gray dashed).}
\label{fig:exp4f}
\end{figure}

\subsection{What did not work and why}\label{sec:failures}
Every configuration above was preceded by runs that failed. We group them by mechanism; the per-run numbers are in the experiment report (Appendix~\ref{app:report}).

\paragraph{The bifurcation barrier.}
Moving a reference eigenvalue from $-0.5$ to $+1$ requires passing through $0$, where the reference is not hyperbolic. Left alone on the saddle, the optimizer avoided the crossing: without an anchor the eigenvalues ended at $\expTwoLamF$ and the reference equilibrium was mapped to $F_\theta^{-1}(0)=\expTwoFpF$, far outside the box, so that a distant sink whose slow direction mimics the unstable direction matched $f$ on the box to a relative error of $\expTwoVfF$ with the wrong phase portrait (trajectory error $\expTwoErrF$). Pinning the equilibrium with $F_\theta(0)=0$ removes this minimum and the eigenvalue crosses zero within the first few hundred iterations (Figure~\ref{fig:exp2}(d)). Initializing from $Df(x^*)$ when it is known avoids the crossing altogether.

\paragraph{Degenerate Jacobians and the circumvented anchor.}
With the equilibrium pinned but the Jacobian free, $\lambda_1$ crossed zero and stalled at $\expTwoLamA$ while the equation loss fell to $\expTwoLossA$. By Proposition~\ref{prop:flow}(ii) the spectrum of the learned field at the equilibrium \emph{is} $\lambda$, so the residual cannot vanish near $0$; the optimizer instead made $DF_\theta(0)$ nearly singular, confining the mismatch to a neighborhood too small for uniform collocation points to see: $F_\theta^{-1}(0)=\expTwoFpA$ although the anchor residual was negligible, and the true trajectories are squeezed into a thin strip in the learned coordinates (Figure~\ref{fig:exp2}(c)); the trajectory error $\expTwoErrA$ is worse than that of the free run. In two dimensions the Jacobian anchor $DF_\theta(0)=I$ removed this. In three dimensions it did not: with $F_\theta(0)=0$, $DF_\theta(0)=I$ and the plain residual the spectrum reached only $\{\expSevenEigJ\}$ (deviation $\expSevenEigErrJ$; the imaginary part $0.84$ instead of $2$), with equation loss $\expSevenLossJ$ and trajectory error $\expSevenErrJ$. The trained model satisfied the anchor at the point and changed its Jacobian to order one within a ball of radius $10^{-2}$ that contains a negligible fraction of uniform collocation points; in that ball the relative residual is of order one, outside it about $0.1$. Without any anchor the three-dimensional spectrum was $\{\expSevenEigF\}$ (deviation $\expSevenEigErrF$, trajectory error $\expSevenErrF$). A \emph{ball anchor} $F_\theta(x)=x$ on $|x|<0.1$ removed the circumvention ($\|DF_\theta-I\|=\expSevenJacDevSmallB$ at $|x|=10^{-2}$) and moved the spectrum to $\{\expSevenEigB\}$ (deviation $\expSevenEigErrB$, trajectory error $\expSevenErrB$), but it is inconsistent with the exact conjugacy by $O(|x|^2)$ and the relative residual at $|x|=10^{-3}$ was still $\expSevenRelTinyB$. Only R1 fixed it (Section~\ref{sec:sinks}). Two observations generalize. First, by the plain equation loss on uniform points the Jacobian-anchored, ball-anchored and R1 models are nearly indistinguishable ($\expSevenLossJ$, $\expSevenLossUniB$, $\expSevenLossUniR$) while their trajectory errors differ by a factor of $25$: the uniform loss is the wrong metric for invariants that live at a point. Second, in six dimensions the plain residual with point anchors failed identically, with one complex pair instead of two ($\{\expEightEigP\}$, deviation $\expEightEigErrP$, trajectory error $\expEightErrP$, nine times that of R1), a relative residual of $\expEightRelTinyP$ at $|x|=10^{-3}$ against $\expEightRelTinyR$ for R1, a Jacobian deviation of $\expEightJacDevMidP$ at $|x|=10^{-1}$ against $\expEightJacDevMidR$, and a uniform loss ($\expEightLossUniP$) within $20\%$ of R1's ($\expEightLossUniR$).

\paragraph{The orientation barrier.}
Remark~\ref{rem:orientation} was learned the hard way. The Van der Pol cycle is traversed clockwise, i.e.\ $\omega<0$ in the convention of Section~\ref{sec:reference}, and the first four Van der Pol runs initialized both frequencies at $+1$. The optimizer had to drive them through zero; in the first run it never did and the model rotated the wrong way, in later runs it did and the frequencies settled with the right sign but at $0.15$--$0.18$ in magnitude, four to five times too slow (final invariants $(\expFourDInv)$ in the fourth run), which a transformation with enough capacity compensated on short-lag pairs by distorting the coordinates (loss $\expFourDLossOneStep$ for single steps, $\expFourDLossLagHalf$ for lags up to $0.5$, $\expFourDLossLagFour$ for lags up to $4$; trajectory error $\expFourDErr$). In the four-dimensional Hamiltonian experiment the first run initialized $\omega_0=+1$; one frequency changed sign and the other did not, the equation loss stalled at $\expNineLossW$, and the rollouts left the energy surface (trajectory error $\expNineErrW$, energy error up to $\expNineEnergyW$; gray curves in Figure~\ref{fig:exp9}(d,e)). In hindsight this also explains the divergent RealNVP seed of the Duffing experiment: its frequency stayed positive with an orientation-preserving $F_\theta$, while the symplectic run of that experiment had driven $\omega(I)$ through zero. The remedy is R3 and costs nothing.

\paragraph{The data-only limit cycle.}
This problem took seven runs and remains open. The first two runs, with trajectories from $[-3,3]^2$ and pairs up to two time units apart, stalled three orders of magnitude above the noise floor (data losses $\expFourCLoss$ with the coupled and $\expFourQLoss$ with the decoupled reference, trajectory errors $\expFourCErr$ and $\expFourQErr$); the latent view showed the true trajectories folded onto a band, not spiraling onto a circle, so the model was an interpolant of short pairs that the exact group law extrapolated to nonsense. A lag curriculum with the point anchor brought the short-lag loss down by an order of magnitude but the long lags undid it; the coupled reference gave up its cycle ($a<0$, invariants $(\expFourBCInv)$), as Remark~\ref{rem:consequences} predicts for a Floquet exponent locked to the origin rate, while the decoupled one kept it. A longer curriculum with a neighborhood anchor plateaued at $\expFourCurPlateauA$ on the shortest lags, $600$ times the noise floor, \emph{before any long lag was introduced}, and a latent-space loss added to it collapsed the interior of the cycle in latent space (latent term $\expFourLatLossLatent$, original-space loss $\expFourLatLoss$), the point-anchor degeneracy in a new guise. The twelve-block scale-bound-$1$ transformation lowered the short-lag plateau eightfold, to $\expFourDRunMeanMin$, still $\expFourDRunMeanRatio$ times the noise floor, so that a lag lengthened by criterion was never lengthened. Initializing the invariants from the data (R3) was the first change that worked: the criterion fired, the cycle frequency was learned to $\expFourEFInvErrWcDec\%$ and the rollouts followed the cycle, but the focus and Floquet invariants drifted \emph{away} from their good initial estimates to $\expFourEFInvErrRate$--$\expFourEFInvErrFloquet\%$ off, because the loss on short pairs near the cycle does not see them. Stratified pairs recovered the focus rate (Section~\ref{sec:weaker}); an additional long-lag loss whose gradient reached only the reference parameters made everything worse (invariants $(\expFourGLInv)$, $\expFourGLInvErrRate$--$\expFourGLInvErrWo\%$ off), because the reference chased a residual it cannot reduce alone and the short-lag loss re-adapted $F_\theta$ to the moved reference.

The seventh run located the obstacle. With the four invariants \emph{frozen at their true values} and only $F_\theta$ trained under the same protocol, the model fits the data worse, not better: re-evaluated on identical batches, the oracle's loss is $\expFourCmpOStep$ for single steps against $\expFourCmpEStep$ for the data-initialized model with learned invariants and $\expFourCmpGStep$ for the stratified one, $\expFourCmpOLagHalf$ against $\expFourCmpELagHalf$ and $\expFourCmpGLagHalf$ for lags up to $0.5$, and only at lags up to $4$, which no model trained on, is the oracle better than the data-initialized model ($\expFourCmpOLagFour$ against $\expFourCmpELagFour$; the stratified model has $\expFourCmpGLagFour$); the in-box rollout errors are $\expFourCmpOInBox$, $\expFourCmpEInBox$ and $\expFourCmpGInBox$. With the invariants pinned, $F_\theta$ has to realize the exact conjugacy between the Van der Pol flow and the reference, and this coupling network trained on short-lag pairs cannot do it to better than a short-lag loss of order $10^{-3}$, twenty times the noise floor; when the invariants are free, the optimizer trades a wrong Floquet exponent for a transformation it can fit, and the data loss, which only sees pairs at most one time unit apart, rewards the trade. The obstacle is therefore the expressivity of the transformation on a relaxation-type cycle, not the identifiability of the Floquet exponent, and the invariants estimated from the data are more trustworthy than their trained values.

\section{Related work}\label{sec:related}
\emph{Conjugation ansatz.} \citet{bizzi2025ncf} propose neural conjugate flows, which are \eqref{eq:ansatz} with an affine reference and a coupling-layer $F$; they prove universality of the affine case after augmentation and report that a fixed affine reference cannot adapt to a limit cycle. Proposition~\ref{prop:invariants} explains that observation, Section~\ref{sec:reference} provides references with the required invariants, and their initialization of $A$ from the Jacobian is one instance of R3. What our experiments add to that line of work is the finding that even with the right invariants available, a plain residual lets the optimizer hide a spectral mismatch next to the equilibrium, and the recipe R1 that prevents it. \emph{Koopman and diffeomorphic linearization.} Koopman-based methods learn coordinates in which the dynamics is linear \citep{budisic2012applied,lusch2018deep}; \citet{bevanda2022diffeomorphic} use a normalizing flow as the diffeomorphism and the Jacobian at the equilibrium as the linear part, which is \eqref{eq:ansatz} with $\Psi_t=e^{At}$ and $A$ fixed. Our linear reference learns $A$, and the centralizer analysis says what the Jacobian anchor removes. \emph{Conjugate mappings.} \citet{bramburger2021conjugate} learn conjugacies of Poincar\'e maps to low-dimensional normal forms with autoencoders; the present work is the continuous-time, invertible-network analog, with the invariants made explicit. \emph{Action--angle networks.} Learned canonical transformations to action--angle variables \citep{li2020neural,bondesan2019learning,daigavane2022learning} are the symplectic special case of our action--angle reference with $d\geq1$; Lemma~\ref{lem:gradient} states the condition under which a learnable multivariate frequency map keeps the reference Hamiltonian, and Experiment~9 shows that a symplectic transformation recovers the actions where a non-symplectic one with the same frequencies does not. \emph{Flow representations.} Neural Flows \citep{bilos2021neuralflows} parametrize $\Phi_\theta(t,x)$ directly and give up the group law; neural ODEs \citep{chen2018neuralode} and learned time-$h$ maps \citep{qin2019data} keep it by integration or iteration. SympNets \citep{jin2020sympnets} and Hamiltonian neural networks \citep{greydanus2019hnn} preserve the symplectic structure of a time-$h$ map or of the field; our symplectic transformation uses SympNet modules as $F_\theta$ and obtains an exactly symplectic flow for every $t$. Experiment~1 is the honest comparison with the time-input alternative: on a problem with no recurrence the direct parametrization is more accurate inside its horizon, and the conjugacy representation pays for its structure in optimization difficulty.

\section{Limitations and outlook}\label{sec:limitations}
\emph{Scope of the evidence.} The three claims we make as recipes rest on three seeds each on one three-dimensional, one six-dimensional and one four-dimensional system with a known field. The other successes are single runs, and one (the pendulum) showed run-to-run variation in exactly the quantity R1 is designed to pin. Nothing has been tried beyond six dimensions or on a non-polynomial field. \emph{Optimization.} The transformation is harder to fit than a direct time-input network on the easiest problem we tried, and the data-only limit cycle is not solved; the oracle run says the limiting factor there is the expressivity of a coupling network on a relaxation-type cycle. Architectures with more local expressivity, an equation residual on whatever part of the field is known, or a loss on pairs several periods apart, which the group law makes as cheap as consecutive pairs, are the next things to try. \emph{Global charts.} A single $F_\theta$ represents the dynamics on one region conjugate to the reference; systems with several attractors or separatrices need a reference with the same global structure or an atlas of transformations on invariant sub-domains, across which the group law survives because basins are invariant. \emph{Chaos and resonance.} Chaotic attractors are not conjugate to any explicitly integrable flow and the method is not intended for them; near resonant equilibria the reference must contain the resonant normal-form terms, which can be added while keeping the flow explicit. \emph{Identifiability.} With data only, the invariants are identifiable in principle but, as the Van der Pol runs show, not necessarily through a loss on short-lag pairs; $F_\theta$ is identifiable only modulo the centralizer, and anchors must be chosen for the quantities of interest.

\section*{Conclusion}
A neural vector field transformation represents the flow of an autonomous system as a conjugate of an explicitly integrable reference flow. The representation is a flow for every parameter value, inherits the phase portrait of the reference and admits a time-uniform error bound; all of this holds exactly in every run. The reference must match all $C^1$-conjugacy invariants of the target, including the sense of rotation, and parametrized reference flows with explicit integral curves can expose them. Whether the method works is then decided by the training: with a relative residual on scale-free collocation points and point anchors, with a symplectic transformation and a gradient-form frequency map, and with invariants initialized from the field or the data, the spectrum of a six-dimensional sink and the action--angle variables of a four-dimensional Hamiltonian system are recovered reproducibly; without these devices the optimizer finds transformations that hide the mismatch where it is not sampled, and the training loss does not notice. The open problem is the transformation itself: its expressivity on relaxation-type dynamics and its conditioning on problems a direct network solves more easily.

\section*{Reproducibility}
The code for all experiments, including the reference flows of Section~\ref{sec:reference}, the trained models, the result files from which every number in this paper is generated, and the full experiment report are available from the author on request; every figure is produced by one script.

\appendix
\section{Experimental details}\label{app:details}
Table~\ref{tab:details} lists the configuration of each final run. Collocation batches are $2\,048$ points per step throughout. The scale bound is the amplitude of the $\tanh$ bounding the log-scale of each coupling block.
\begin{table}[h]
\caption{Final configurations. ``R1'' means relative residual, log-radial sampling of half the batch in $|x|\in[10^{-3},1]$, anchors $F_\theta(0)=0$, $DF_\theta(0)=I$.}
\label{tab:details}
\centering\scriptsize
\begin{tabular}{lp{3.2cm}p{4.2cm}p{3.6cm}p{2.1cm}}
\toprule
\# & $F_\theta$ & loss & domain / data & steps (Adam + L-BFGS) \\
\midrule
1 & 12 coupling blocks, width 64, scale bound 1 & push-forward residual; section anchor $\{x_1=0\}$ & $[-3,3]\times[-4.5,4.5]$ & $15\,000+1\,000$ \\
2 & 6 coupling blocks, width 64 & push-forward residual; anchors $F_\theta(0)=0$, $DF_\theta(0)=I$ & $[-1.5,1.5]^2$ & $6\,000+500$ \\
3 & 6 coupling blocks, width 64 & push-forward residual & $E<1.8$ & $6\,000+500$ \\
4 & 12 coupling blocks, width 64, scale bound 1 & data loss \eqref{eq:ldat}, stratified pairs, lag $0.5$ doubled by criterion; anchor $F_\theta(0)=0$; invariants from data, frozen 3\,000 steps & 40 trajectories, length 12, $\Delta t=0.05$, noise $0.005$, from $[-2.5,2.5]^2$ & $10\,000+500$ \\
5 & 6 coupling blocks, width 64 & data loss + residual for $f(\cdot;\alpha)$ & 12 trajectories, length 15, $\Delta t=0.1$, noise $\expFiveNoise$ & $8\,000+300$ \\
6 & 8 gradient modules, width 64 (RealNVP: 6 blocks) & push-forward residual & $H<3$ & $6\,000+500$ \\
7 & 9 coupling blocks, width 64 & R1 & $[-1,1]^3$ & $8\,000+500$ \\
8 & 12 coupling blocks, width 64 ($\expEightParamsR$ parameters) & R1 & $[-1,1]^6$ & $6\,000+500$ \\
9 & 8 gradient modules, width 64 ($\expNineParamsS$ parameters) & push-forward residual; $\omega=\omega_0+\nabla K$, $\omega_0=-1$ & $[-1.6,1.6]^4$, $H<1.5$ & $6\,000+500$ \\
\bottomrule
\end{tabular}
\end{table}

\section{The experiment report}\label{app:report}
Every run mentioned in Section~\ref{sec:failures}, and several that are not (variants of Experiment~1, the retrained pendulum copy, the two converged RealNVP seeds of Experiment~6, the ball-anchored three-dimensional run), is documented with its full configuration, its numbers and the diagnosis made at the time in the experiment report that is available with the code. That report was written chronologically as the runs happened; this paper reorganizes its content by mechanism and reports only the final configuration of each experiment. The numbers in both documents are generated from the same result files by the same script, so they agree.

\bibliographystyle{unsrtnat}
\bibliography{references}

\begin{thebibliography}{30}
\providecommand{\natexlab}[1]{#1}
\providecommand{\url}[1]{\texttt{#1}}
\expandafter\ifx\csname urlstyle\endcsname\relax
  \providecommand{\doi}[1]{doi: #1}\else
  \providecommand{\doi}{doi: \begingroup \urlstyle{rm}\Url}\fi

\bibitem[Lagaris et~al.(1998)Lagaris, Likas, and
  Fotiadis]{lagaris1998artificial}
Isaac~E. Lagaris, Aristidis Likas, and Dimitrios~I. Fotiadis.
\newblock Artificial neural networks for solving ordinary and partial
  differential equations.
\newblock \emph{IEEE Transactions on Neural Networks}, 9\penalty0 (5):\penalty0
  987--1000, 1998.

\bibitem[Raissi et~al.(2019)Raissi, Perdikaris, and
  Karniadakis]{raissi2019pinn}
Maziar Raissi, Paris Perdikaris, and George~Em Karniadakis.
\newblock Physics-informed neural networks: A deep learning framework for
  solving forward and inverse problems involving nonlinear partial differential
  equations.
\newblock \emph{Journal of Computational Physics}, 378:\penalty0 686--707,
  2019.

\bibitem[Qin et~al.(2019)Qin, Wu, and Xiu]{qin2019data}
Tong Qin, Kailiang Wu, and Dongbin Xiu.
\newblock Data driven governing equations approximation using deep neural
  networks.
\newblock \emph{Journal of Computational Physics}, 395:\penalty0 620--635,
  2019.

\bibitem[Chen et~al.(2018)Chen, Rubanova, Bettencourt, and
  Duvenaud]{chen2018neuralode}
Ricky T.~Q. Chen, Yulia Rubanova, Jesse Bettencourt, and David Duvenaud.
\newblock Neural ordinary differential equations.
\newblock In \emph{Advances in Neural Information Processing Systems},
  volume~31, 2018.

\bibitem[Wang and Yu(2021)]{wang2021physics}
Rui Wang and Rose Yu.
\newblock Physics-guided deep learning for dynamical systems: A survey, 2021.
\newblock arXiv:2107.01272.

\bibitem[Bilo{\v{s}} et~al.(2021)Bilo{\v{s}}, Sommer, Rangapuram, Januschowski,
  and G{\"u}nnemann]{bilos2021neuralflows}
Marin Bilo{\v{s}}, Johanna Sommer, Syama~Sundar Rangapuram, Tim Januschowski,
  and Stephan G{\"u}nnemann.
\newblock Neural flows: Efficient alternative to neural {ODE}s.
\newblock In \emph{Advances in Neural Information Processing Systems},
  volume~34, 2021.

\bibitem[Jin et~al.(2020)Jin, Zhang, Zhu, Tang, and
  Karniadakis]{jin2020sympnets}
Pengzhan Jin, Zhen Zhang, Aiqing Zhu, Yifa Tang, and George~Em Karniadakis.
\newblock {SympNets}: Intrinsic structure-preserving symplectic networks for
  identifying {H}amiltonian systems.
\newblock \emph{Neural Networks}, 132:\penalty0 166--179, 2020.

\bibitem[Arnold(1992)]{arnold1992ode}
Vladimir~I. Arnold.
\newblock \emph{Ordinary Differential Equations}.
\newblock Springer-Verlag, Berlin, 1992.

\bibitem[Lee(2012)]{lee2012smooth}
John~M. Lee.
\newblock \emph{Introduction to Smooth Manifolds}, volume 218 of \emph{Graduate
  Texts in Mathematics}.
\newblock Springer, 2 edition, 2012.

\bibitem[Palis and de~Melo(1982)]{palis1982geometric}
Jacob Palis and Welington de~Melo.
\newblock \emph{Geometric Theory of Dynamical Systems: An Introduction}.
\newblock Springer-Verlag, New York, 1982.

\bibitem[Bizzi et~al.(2025)Bizzi, Nissenbaum, and Pereira]{bizzi2025ncf}
Arthur Bizzi, Lucas Nissenbaum, and Jo{\~a}o~M. Pereira.
\newblock Neural conjugate flows: A physics-informed architecture with flow
  structure.
\newblock In \emph{Proceedings of the AAAI Conference on Artificial
  Intelligence}, 2025.
\newblock arXiv:2411.08326.

\bibitem[Arnold(1989)]{arnold1989mathematical}
Vladimir~I. Arnold.
\newblock \emph{Mathematical Methods of Classical Mechanics}.
\newblock Springer-Verlag, New York, 2 edition, 1989.

\bibitem[Bhatia and Szeg{\H{o}}(1970)]{bhatia1970stability}
Nam~P. Bhatia and Giorgio~P. Szeg{\H{o}}.
\newblock \emph{Stability Theory of Dynamical Systems}.
\newblock Springer-Verlag, Berlin, 1970.

\bibitem[Hartman(1960)]{hartman1960lemma}
Philip Hartman.
\newblock A lemma in the theory of structural stability of differential
  equations.
\newblock \emph{Proceedings of the American Mathematical Society}, 11\penalty0
  (4):\penalty0 610--620, 1960.

\bibitem[Sternberg(1957)]{sternberg1957local}
Shlomo Sternberg.
\newblock Local contractions and a theorem of {P}oincar\'e.
\newblock \emph{American Journal of Mathematics}, 79\penalty0 (4):\penalty0
  809--824, 1957.

\bibitem[Arnold(1988)]{arnold1988geometrical}
Vladimir~I. Arnold.
\newblock \emph{Geometrical Methods in the Theory of Ordinary Differential
  Equations}.
\newblock Springer-Verlag, New York, 2 edition, 1988.

\bibitem[Dinh et~al.(2015)Dinh, Krueger, and Bengio]{dinh2015nice}
Laurent Dinh, David Krueger, and Yoshua Bengio.
\newblock {NICE}: Non-linear independent components estimation, 2015.
\newblock arXiv:1410.8516.

\bibitem[Dinh et~al.(2017)Dinh, Sohl-Dickstein, and Bengio]{dinh2017realnvp}
Laurent Dinh, Jascha Sohl-Dickstein, and Samy Bengio.
\newblock Density estimation using {Real NVP}.
\newblock In \emph{International Conference on Learning Representations}, 2017.

\bibitem[Jacobsen et~al.(2018)Jacobsen, Smeulders, and
  Oyallon]{jacobsen2018irevnet}
J{\"o}rn-Henrik Jacobsen, Arnold Smeulders, and Edouard Oyallon.
\newblock i-{RevNet}: Deep invertible networks.
\newblock In \emph{International Conference on Learning Representations}, 2018.

\bibitem[Kingma et~al.(2016)Kingma, Salimans, Jozefowicz, Chen, Sutskever, and
  Welling]{kingma2016improving}
Diederik~P. Kingma, Tim Salimans, Rafal Jozefowicz, Xi~Chen, Ilya Sutskever,
  and Max Welling.
\newblock Improving variational inference with inverse autoregressive flow.
\newblock In \emph{Advances in Neural Information Processing Systems},
  volume~29, 2016.

\bibitem[Teshima et~al.(2020)Teshima, Ishikawa, Tojo, Oono, Ikeda, and
  Sugiyama]{teshima2020coupling}
Takeshi Teshima, Isao Ishikawa, Koichi Tojo, Kenta Oono, Masahiro Ikeda, and
  Masashi Sugiyama.
\newblock Coupling-based invertible neural networks are universal
  diffeomorphism approximators.
\newblock In \emph{Advances in Neural Information Processing Systems},
  volume~33, 2020.

\bibitem[Chen et~al.(2024)Chen, Yang, Xiang, and Hao]{chen2024autodiff}
Chuqi Chen, Yahong Yang, Yang Xiang, and Wenrui Hao.
\newblock Automatic differentiation is essential in training neural networks
  for solving differential equations, 2024.
\newblock arXiv:2405.14099.

\bibitem[Budi{\v{s}}i{\'c} et~al.(2012)Budi{\v{s}}i{\'c}, Mohr, and
  Mezi{\'c}]{budisic2012applied}
Marko Budi{\v{s}}i{\'c}, Ryan Mohr, and Igor Mezi{\'c}.
\newblock Applied {K}oopmanism.
\newblock \emph{Chaos}, 22\penalty0 (4):\penalty0 047510, 2012.

\bibitem[Lusch et~al.(2018)Lusch, Kutz, and Brunton]{lusch2018deep}
Bethany Lusch, J.~Nathan Kutz, and Steven~L. Brunton.
\newblock Deep learning for universal linear embeddings of nonlinear dynamics.
\newblock \emph{Nature Communications}, 9:\penalty0 4950, 2018.

\bibitem[Bevanda et~al.(2022)Bevanda, Kirmayr, Sosnowski, and
  Hirche]{bevanda2022diffeomorphic}
Petar Bevanda, Johannes Kirmayr, Stefan Sosnowski, and Sandra Hirche.
\newblock Learning the {K}oopman eigendecomposition: A diffeomorphic approach.
\newblock In \emph{American Control Conference (ACC)}, 2022.
\newblock arXiv:2110.07786.

\bibitem[Bramburger et~al.(2021)Bramburger, Brunton, and
  Kutz]{bramburger2021conjugate}
Jason~J. Bramburger, Steven~L. Brunton, and J.~Nathan Kutz.
\newblock Deep learning of conjugate mappings.
\newblock \emph{Physica D: Nonlinear Phenomena}, 427:\penalty0 133008, 2021.

\bibitem[Li et~al.(2020)Li, Dong, Zhang, and Wang]{li2020neural}
Shuo-Hui Li, Chen-Xiao Dong, Linfeng Zhang, and Lei Wang.
\newblock Neural canonical transformation with symplectic flows.
\newblock \emph{Physical Review X}, 10:\penalty0 021020, 2020.

\bibitem[Bondesan and Lamacraft(2019)]{bondesan2019learning}
Roberto Bondesan and Austen Lamacraft.
\newblock Learning symmetries of classical integrable systems, 2019.
\newblock arXiv:1906.04645.

\bibitem[Daigavane et~al.(2022)Daigavane, Kosmala, Cranmer, Smidt, and
  Ho]{daigavane2022learning}
Ameya Daigavane, Arthur Kosmala, Miles Cranmer, Tess Smidt, and Shirley Ho.
\newblock Learning integrable dynamics with action-angle networks, 2022.
\newblock arXiv:2211.15338.

\bibitem[Greydanus et~al.(2019)Greydanus, Dzamba, and
  Yosinski]{greydanus2019hnn}
Samuel Greydanus, Misko Dzamba, and Jason Yosinski.
\newblock Hamiltonian neural networks.
\newblock In \emph{Advances in Neural Information Processing Systems},
  volume~32, 2019.

\end{thebibliography}

\end{document}